\documentclass[12pt,A4paper]{article}
\usepackage{amssymb,amsmath,amsthm,hyperref}
\usepackage[mathscr]{eucal}
\usepackage[PS]{diagrams}

\newcommand{\K}{\ensuremath{\mathcal{K}}}
\newcommand{\Ks}{\ensuremath{\mathcal{K}^{\strong}}}

\newcommand{\Kso}{\ensuremath{\K^{\strong}_{<\aleph_0}}}

\newcommand{\Ss}{\ensuremath{\mathcal{S}}} 
\newcommand{\Smax}{\ensuremath{S^{\mathrm{max}}}}

\newcommand{\DKD}{\ensuremath{\Der(K/D_0)}}
\newcommand{\OKD}{\ensuremath{\Omega(K/D_0)}}

\DeclareMathOperator{\Rot}{Rot}  

\newenvironment{step}[1][Step]
    {\medskip \noindent{\slshape\bfseries #1}\\}
    {}

\newcommand{\Step}[1]{\paragraph{Step #1}}

\DeclareMathOperator{\tp}{tp}  
\DeclareMathOperator{\td}{td}  
\DeclareMathOperator{\ldim}{ldim}  
\DeclareMathOperator{\codim}{codim}  

\DeclareMathOperator{\rk}{rk}
\DeclareMathOperator{\grk}{grk} 
\DeclareMathOperator{\Der}{Der}  
\DeclareMathOperator{\Aut}{Aut}  
\DeclareMathOperator{\cl}{cl}   
\DeclareMathOperator{\acl}{acl}   
\DeclareMathOperator{\loc}{Loc}   
\DeclareMathOperator{\Ann}{Ann}   

\DeclareMathOperator{\Mat}{Mat}  
\DeclareMathOperator{\Jac}{Jac}  

\DeclareMathOperator{\logd}{\mathnormal{lD}}  

\DeclareMathOperator{\pr}{pr}  

\newcommand{\alg}{\ensuremath{\mathrm{alg}}} 

\newcommand{\restrict}[1]{\ensuremath{\!\!\upharpoonright_{#1}}}

\newcommand{\N}{\ensuremath{\mathbb{N}}}
\newcommand{\Z}{\ensuremath{\mathbb{Z}}}
\newcommand{\Q}{\ensuremath{\mathbb{Q}}}

\newcommand{\C}{\ensuremath{\mathbb{C}}}

\newcommand{\Cat}{\ensuremath{\mathcal{C}}} 
\renewcommand{\L}{\ensuremath{\mathcal{L}}} 
\newcommand{\Loo}{\ensuremath{\mathcal{L}_{\omega_1,\omega}}}

\newcommand{\tensor}{\otimes}

\newcommand{\ga}{\ensuremath{{\mathbb{G}_\mathrm{a}}}}   
\newcommand{\gm}{\ensuremath{{\mathbb{G}_\mathrm{m}}}}  

\newcommand{\HV}{\ensuremath{\mathcal{H}_V}} 
\newcommand{\Jf}[1]{\ensuremath{\mathcal{J}_{#1}}} 

\renewcommand{\phi}{\varphi}
\renewcommand{\le}{\ensuremath{\leqslant}}
\renewcommand{\ge}{\ensuremath{\geqslant}}
\newcommand{\tuple}[1]{\ensuremath{\langle #1 \rangle}}
\newcommand{\class}[2]{\ensuremath{\left\{ #1 \,\left|\, #2 \right.\right\}}}

\newcommand{\iso}{\cong}

\newcommand{\into}{\hookrightarrow}

\newcommand{\subs}{\subseteq} 
\newcommand{\sups}{\supseteq} 

\DeclareMathOperator{\fin}{fin}  
\newcommand{\subsetfin}{\ensuremath{\subseteq_{\fin}}} 
\newcommand{\finsub}{\subsetfin}

\newcommand{\subsfg}{\ensuremath{\subs_{f.g.}}}
\newcommand{\fgsub}{\subsfg}

\newcommand{\minus}{\ensuremath{\smallsetminus}}
\newcommand{\powerset}{\ensuremath{\mathcal{P}}} 

\newcommand{\strong}{\ensuremath{\lhd}} 
\newcommand{\nstrong}{\ensuremath{\not\kern-4pt\lhd\;}} 

\newcommand{\gen}[1]{\ensuremath{\langle #1 \rangle}}   
\newcommand{\hull}[1]{\ensuremath{\lceil #1\rceil}}

\newcommand{\cross}{\ensuremath{\times}}

\newcommand{\ra}[3]{\ensuremath{#1 \stackrel{#2}{\longrightarrow} #3}}

\newcommand{\map}[5]{
\begin{eqnarray*}
#1 & \stackrel{#2}{\longrightarrow} & #3\\ #4 & \longmapsto & #5
\end{eqnarray*}}

\newcommand{\leteq}{\mathrel{\mathop:}=}

\newarrowmiddle{normal}\lhtriangle\rttriangle\uhtriangle\dhtriangle

\newarrowmiddle{iso}{\cong}{\cong}{}{}   

\newarrow{Partial}{-}{-}{-}{-}{harpoon}     

\newarrow{Equal}=====     

\newarrow{Strong}{hooka}-{normal}->     

\newarrow{Str}{}{}{normal}{}{}     

\newarrow{Iso}--{iso}->     

\newtheorem{prop}{Proposition}[section]
\newtheorem{cor}[prop]{Corollary}
\newtheorem{theorem}[prop]{Theorem}
\newtheorem{lemma}[prop]{Lemma}

\newtheorem*{thm}{Theorem}

\theoremstyle{definition}
\newtheorem{defn}[prop]{Definition}

\title{The theory of the exponential differential equations of
  semiabelian varieties}

\author{Jonathan Kirby\thanks{University of Oxford and University of
    Illinois at Chicago, supported by the EPSRC fellowship EP/D065747/1}}

\date{Version 3.1, \today}

\begin{document}

\maketitle

\begin{abstract}
  The complete first order theories of the exponential differential
  equations of semiabelian varieties are given. It is shown that these
  theories also arise from an amalgamation-with-predimension
  construction in the style of Hrushovski. The theories include
  necessary and sufficient conditions for a system of equations to
  have a solution. The necessary conditions generalize Ax's
  differential fields version of Schanuel's conjecture to semiabelian
  varieties.  There is a purely algebraic corollary, the ``Weak CIT''
  for semiabelian varieties, which concerns the intersections of
  algebraic subgroups with algebraic varieties.
\end{abstract}

\tableofcontents

\section{Introduction}

\subsection{The exponential differential equation}

Let $\tuple{F;+,\cdot,D}$ be a differential field of characteristic zero, and consider the \emph{exponential differential equation} $Dx = \frac{Dy}{y}$. If $F$ is a field of meromorphic functions in a variable $t$, with $D$ being $\frac{d}{dt}$, then this is the differential equation satisfied by any $x(t), y(t) \in F$ such that $y(t) = e^{x(t)}$.

James Ax proved the following differential fields version of Schanuel's conjecture.
\begin{theorem}[Ax, \cite{Ax71}]\label{Ax theorem}
  Let $F$ be a field of characteristic zero, $D$ be a derivation on
  $F$ and $C$ be the constant subfield. Suppose $n \ge 1$ and
  $x_1,y_1,\ldots,x_n,y_n \in F$ are such that $Dx_i =
  \frac{Dy_i}{y_i}$ for each $i$, and the $Dx_i$ are $\Q$-linearly
  independent. Then $\td(x_1,y_1,\ldots,x_n,y_n/C) \ge n+1$.
\end{theorem}
 Here and throughout this paper, $\td(X/C)$ means the transcendence degree of the field extension $C(X)/C$. 

 The theorem can be viewed as giving a restriction on the systems of equations which have solutions in a differential field. In this paper it is shown that Ax's theorem is the only restriction on a system of instances of the exponential differential equation and polynomial equations having solutions in a differential field. This is done by proving a matching \emph{existential closedness} theorem, stating that certain systems of equations do have solutions when the field $F$ is differentially closed. The theorem says roughly that certain algebraic subvarieties $V \subs F^{2n}$, whose images under projections are suitably large and which we call \emph{rotund subvarieties}, must have nonempty intersection with the set of $n$-tuples of solutions of the exponential differential equation.

\subsection{Semiabelian varieties}

 We consider not just the exponential differential equation given above, but also the exponential differential equations of every semiabelian variety defined over the field of constants. Next we explain what these equations are, and the language we use to study them. The foundations of algebraic geometry we use are those standard in model theory, which can be found in \cite{Pillay98} or in \cite{MarkerMK}. In particular, we generally work with a fixed algebraically closed field $F$ (always of characteristic zero in this paper), and identify a variety over $F$ with its set of $F$-points, although we may write the latter also as $V(F)$. If $K$ is another field with $K \subs F$ or $F \subs K$ then we write $V(K)$ for the $K$-points of $V$.

 In characteristic zero, we can define a \emph{semiabelian variety} $S$ to be a connected, commutative algebraic group, with no algebraic subgroup isomorphic to the additive group $\ga$. By Chevalley's theorem \cite[p40]{Serre88}, every algebraic group $G$ can be given as an extension
 \[0 \to L \to G \to A \to 0\]
 in a unique way, where $A$ is an abelian variety (a connected projective algebraic group) and $L$ is a linear group. If $G$ is connected and commutative and the characteristic is zero then $L$ is of the form $\ga^l \cross \gm^k$ for some natural numbers $l$ and $k$ \cite[p40, p171]{Serre88}. So $G$ is a semiabelian variety when, in addition, $l = 0$. Special cases include the multiplicative group $\gm$, its powers which are called algebraic tori, and elliptic curves which are one-dimensional abelian varieties. Algebraic groups isomorphic to $\ga^n$ for some natural number $n$ are called \emph{vector groups}.

We need the notions of tangent bundles and the logarithmic derivative map. A good exposition is given in \cite{MarkerMK}, so we just summarize the essential properties we need. Given any connected commutative algebraic group $G$, its tangent bundle $TG$ is also a connected commutative algebraic group. We write $LG$ for the tangent space at the identity of $G$. (The $L$ here stands for Lie algebra, but since the group is commutative, the Lie bracket is trivial.) $LG$ is a vector group with $\dim LG = \dim G$, and $TG$ is canonically isomorphic to $LG \cross G$ as an algebraic group.

 For any differential field $\tuple{F;+,\cdot,D}$ and any commutative algebraic group $G$ defined over the subfield of constants $C$, there is a \emph{logarithmic derivative map}, which is a group homomorphism $\logd_G : G(F) \to LG(F)$. 

 If $G$ is a vector group then $LG$ is canonically isomorphic to $G$. In particular, for any $G$, $LLG$ is canonically isomorphic to $LG$. We have $\logd_\ga(x) = Dx$ and, identifying $L\gm$ with $\ga$ we have $\logd_\gm(y) = \frac{Dy}{y}$, so the usual exponential differential equation can be written as
\[ \logd_{L\gm}(x) = \logd_{\gm}(y).\]

For a general semiabelian variety $S$, defined over the field of constants $C$, we define the \emph{exponential differential equation of $S$} to be 
\[ \logd_{LS}(x) = \logd_{S}(y) \]
under the canonical identification of $LLS$ and $LS$. The equation defines a differential subvariety of $TS$, which we denote by $\Gamma_S$. That is,
\[\Gamma_S = \class{(x,y) \in LS \cross S}{\logd_{LS}(x) = \logd_{S}(y)}.\]

As mentioned earlier, the usual exponential map satisfies the exponential differential equation of $\gm$. If $S$ is a complex semiabelian variety, we may consider $S(\C)$ as a complex Lie group, and $LS(\C)$ can be identified with its universal covering space. The analytic covering map
\[LS(\C) \stackrel{\exp_S}{\longrightarrow} S(\C)\]
is called the \emph{exponential map} of $S(\C)$, and it can be shown via a Lie theory argument that this map satisfies the exponential differential equation for $S$. This is one motivation for considering these equations.

 Having explained the equations under consideration, we now explain the context in which we study them. Let $\tuple{F;+,\cdot,C,D}$ be a differential field of characteristic 0, with $C$ being the constant subfield. Let $C_0$ be a countable subfield of $C$, and let $\Ss$ be a collection of semiabelian varieties, each defined over $C_0$. Expand $F$ by adding a symbol for $\Gamma_S$ for each $S \in \Ss$ (of appropriate arity to be interpreted as a subset of $TS$) and by adding constant symbols for each element of $C_0$. Then forget the deriviation -- consider the reduct $\tuple{F;+,\cdot,C,(\Gamma_S)_{S
    \in\Ss}, (\hat{c})_{c\in C_0}}$. We call this language $\L_\Ss$.

We will give the complete first-order theory of this reduct, in the case where $\tuple{F;+,\cdot,D}$ is a differentially closed field.

\subsection{Outline of the paper}

 In section~\ref{amalg section} of the paper we take the analogues for semiabelian varieties of Ax's theorem (see below) as a starting point. We observe that they can be seen as stating the positivity of a \emph{predimension function}, as used by Hrushovski \cite{Hru93} to construct his new strongly minimal theories -- theories where there is a particularly simple and powerful dimension theory. This is easiest to see in the original multiplicative group setting. Write $x,y$ for the tuples $x_1,\ldots,x_n$ and $y_1,\ldots,y_n$ in theorem~\ref{Ax theorem}, and define
 \[\delta(x,y) = \td(x,y/C)- \ldim_\Q(x/C)\]
 where the second term is the $\Q$-linear dimension of the span of the images of the $x_i$ in the quotient \Q-vector space $F/C$. Then Ax's theorem is equivalent to the statement that for all tuples $x,y \in F^n$ satisfying the exponential differential equation, either $\delta(x,y) \ge 1$ or all the $x_i$ and $y_i$ lie in $C$.

 Using this predimension function, and its generalisations, we construct abstract $\L_\Ss$-theories $T_\Ss$ via a category-theoretic version of Hrushovski's \emph{amalgamation with predimension} technique.  In particular, we obtain a pregeometry with its associated notion of dimension, and the definition (see \ref{rotund defn}) of the rotund subvarieties of the tangent bundles $TS$, which are those occuring in the existential closedness statements. However, we cannot at this stage of the paper show that $T_\Ss$ is first-order axiomatizable.

 Section~\ref{diff fields section} starts by connecting the logarithmic derivatives with differential forms, and goes on to prove that the analogues of Ax's theorem (which we call \emph{Schanuel properties}) do indeed hold in all differential fields. As in Ax's paper, we prove a statement for many commuting derivations. The simpler statement for just one derivation is as follows.
 \begin{thm}[\ref{Schanuel property}, the Schanuel Property, one
   derivation version]
   Let $F$ be a differential field of characteristic zero, with
   constant subfield $C$. Let $S$ be a semiabelian variety defined over
   $C$, of dimension $n$.

   Suppose that $(x,y) \in \Gamma_S$ and $\td(x,y/C) < n + 1$. Then
   there is a proper algebraic subgroup $H$ of $S$ and a constant point
   $\gamma$ of $TS$ such that $(x,y)$ lies in the coset $\gamma \cdot
   TH$.
 \end{thm}
 We also prove that the existential closedness axioms hold in differentially closed fields.
 \begin{thm}[\ref{reduct satisfies EC}, Existential Closedness]
  Let $F$ be a differentially closed field of characteristic zero, and $S$ a semiabelian variety defined over $C$. Then for each irreducible rotund subvariety $V$ of $TS$, and each parametric family $(W_e)_{e \in Q(C)}$ of proper subvarieties of $V$, with $Q$ a constructible set defined over $C$, there is $g \in \Gamma_S \cap V \minus \bigcup_{e \in Q(C)} W_e$.
 \end{thm}
This theorem extends work of Crampin \cite{Crampin}, who considered a case where the variety $V$ is defined over the constant subfield, just for the multiplicative group.

 In section~\ref{fot section} we apply the compactness theorem of first-order logic to the Schanuel properties proved in section~\ref{diff fields section} to show that they are first-order expressible, and to deduce a result in diophantine geometry, concerning the intersections of algebraic subgroups of semiabelian varieties with algebraic varieties.
\begin{thm}[\ref{weak CIT}, ``Weak CIT'' for semiabelian varieties]
  Let $S$ be a semiabelian variety defined over an algebraically
  closed field $C$ of characteristic zero. Let $(U_p)_{p \in P}$ be a
  parametric family of algebraic subvarieties of $S$. There is a
  finite family $\Jf{U}$ of proper algebraic subgroups of $S$ such
  that, for any coset $\kappa = a \cdot H$ of any algebraic subgroup
  $H$ of $S$ and any $p \in P(C)$, if $X$ is an irreducible component
  of $U_p \cap \kappa$ and
  \[\dim X = (\dim U_p + \dim \kappa - \dim S) + t\]
  with $t>0$, an atypical component of the intersection, then there is
  $J \in \Jf{U}$ of codimension at least $t$ and $s \in S(C)$ such
  that $X \subs s \cdot J$.
\end{thm}
This is a weak version of the \emph{Conjecture on the intersection of algebraic subgroups with subvarieties} stated by Zilber in \cite{Zilber02esesc}, and is the natural generalization to
semiabelian varieties of the version proved there for algebraic tori. (``CIT'' stands for the Conjecture on Intersections with Tori.) For a discussion of results and conjectures of this form, see \cite{Zilber02esesc} and \cite{BMZ07}.

We then use this weak CIT result to show that the notion of rotundity of a subvariety is definable, and hence that the existential closedness property is first-order expressible. Thus the theories $T_\Ss$ are first-order, and we then show they are complete. Finally, we give two simple model-theoretic properties of the $T_\Ss$.

I believe the results of this paper can be generalised to arbitrary
commutative algebraic groups, although vector groups must be treated
separately because their exponential maps are just the identity maps.
Indeed, Bertrand \cite{Bertrand06} has proved the Schanuel property
for commutative algebraic groups with no vectorial quotients, a
generalization of semiabelian varieties. He makes use of another paper
of Ax \cite{Ax72a}, and considers only the case where the differential
field is a field of meromorphic functions. The method of \cite{ASCWE} and of \S5.5 of \cite{DPhil}
generalizes Bertrand's result to any differential field. In these cases, the groups are still defined
over the constant field $C$ (or, essentially equivalently, are
isoconstant). Bertrand and Pillay have also considered Schanuel properties in the
non-isoconstant case \cite{BP08}.

Much of the work of this paper was done as part of my DPhil thesis
\cite{DPhil} under the supervision of Boris Zilber, and his great
influence will be clear to anyone who knows his work.

\section{Amalgamation}\label{amalg section}

In this section we put aside differential fields and construct an
abstract $\L_\Ss$-structure and its theory $T_\Ss$. In
section~\ref{diff fields section} we show that the reducts of
differentially closed fields are models of $T_\Ss$. It is not
immediate that $T_\Ss$ is first-order axiomatizable, but this is
proven in section~\ref{fot section}. We start by giving the universal
part of $T_\Ss$.

\subsection{The universal theory}\label{universal theory}

 Fix a countable field $C_0$ of characteristic zero, and a collection $\Ss$ of semiabelian varieties, each defined over $C_0$. We assume also that $C_0$ is large enough that every algebraic homomorphism between any members of $\Ss$ is defined over $C_0$. For example, if $\Ss$ is the collection $\class{\gm^n}{n \in \N}$ of algebraic tori, then we can just take $C_0 = \Q$. In any case it suffices to take $C_0$ to be algebraically closed. Recall that the language $\L_\Ss$ is $\tuple{+,\cdot,C,(\Gamma_S)_{S \in\Ss}, (\hat{c})_{c\in C_0}}$, the field language augmented by relation symbols for the constant field and for each solution set $\Gamma_S$, and by constant symbols for the elements of $C_0$. The theory $T_\Ss^U$ is given as follows.
\begin{enumerate}
  \item[\textup{U1}] $F$ is an algebraically closed field, $C$ is a
    (relatively) algebraically closed subfield, and the subfield $C_0$
    of $C$ is named by parameters.
  \item[\textup{U2}] For each $S\in\Ss$, $\Gamma_S$ is a subgroup of
    $TS$.
  \item[\textup{U3}] For each $S\in\Ss$, $TS(C) \subs \Gamma_S$ 
  \item[\textup{U4}] $(0, y) \in \Gamma_S \iff y \in S(C)$ and $(x,
    1) \in \Gamma_S \iff x \in LS(C)$, where $0$ is the identity of
    $LS$ and $1$ is the identity of $S$.
  \item[\textup{U5}] If $\ra{S_1}{f}{S_2}$ is an algebraic group
    homomorphism then $(Tf)(\Gamma_{S_1}) \subs \Gamma_{S_2}$, and
    if $f$ is an isogeny then also $\Gamma_{S_1} =
    (Tf)^{-1}(\Gamma_{S_2})$.
  \item[\textup{U6}] For each $S_1, S_2 \in \Ss$, if $S_1 \subs S_2$
    then $\Gamma_{S_1} = \Gamma_{S_2} \cap TS_1$.
  \item[\textup{U7}] For each $S_1, S_2 \in \Ss$, $\Gamma_{S_1
      \cross S_2} = \Gamma_{S_1} \cross \Gamma_{S_2}$.
  \item[\textup{SP}] For each $S \in \Ss$, if $g \in \Gamma_S$ and
    $\td(g/C) < \dim S + 1$ then there is a proper algebraic subgroup
    $H$ of $S$ and $\gamma \in TS(C)$ such that $g$ lies in the coset
    $\gamma\cdot TH$.
\end{enumerate}

\begin{lemma}\label{U axioms are first order}
  The axioms U1 --- U7 can all be expressed as first order axiom
  schemes in the language $\L_\Ss$.
\end{lemma}
\begin{proof}
  This is almost immediate. For U5, recall that by assumption on $C_0$, every algebraic homomorphism $\ra{S_1}{f}{S_2}$ is defined over $C_0$, and
  hence is $\emptyset$-definable in $\L_\Ss$.
\end{proof}
The last axiom, SP, is the Schanuel property. Since each $S\in \Ss$
has only countably many proper algebraic subgroups and there are only
countably many polynomials, it follows that SP can be expressed as a
sentence in the infinitary language $\Loo$. We show later
(corollary~\ref{SP is first order}) that SP can also be expressed as a
first order axiom scheme.

The superscript ``U'' in $T_\Ss^U$ stands for universal. The theory is
universal, that is, if $M$ is a model and $N$ is a substructure of $M$
then $N$ is also a model, with the exception of the part of U1 that
says that the field $F$ is algebraically closed. It will be convenient
to work in a setting in which we only consider substructures whose
underlying field is algebraically closed. In this non-elementary
setting, the theory $T_\Ss^U$ is precisely the ``theory of
substructures''.

If $\Ss$ is not closed under products, then for $S_1,S_2 \in \Ss$ we
can use axiom U7 to define $\Gamma_{S_1 \cross S_2}$. Thus we may
assume that $\Ss$ is closed under products. Similarly, using U6 we may
assume that $\Ss$ is closed under taking (connected) subgroups, and
using U5 we may assume that $\Ss$ is closed under quotients. An
\emph{isogeny} is a surjective homomorphism with finite kernel. Groups
$S_1$ and $S_2$ are said to be \emph{isogenous} iff there is $S_3$ and
isogenies $\ra{S_3}{}{S_1}$ and $\ra{S_3}{}{S_2}$. By U5 we can also
assume that $\Ss$ is closed under the equivalence relation of isogeny.

\subsection{The category \K}

We now use Hrushovski's amalgamation-with-predimension technique to
produce a ``countable universal domain'', $U$, for $T_\Ss^U$. From the
construction of $U$ we will obtain an axiomatization of its complete
theory, $T_\Ss$. Again, it will be clear that the axiomatization is
expressible in $L_{\omega_1\omega}$.  We will later extract the first
order part of the theory.

We apply the amalgamation construction not to the category of all
countable models of $T_\Ss^U$, but to a subcategory. Fix a countable
algebraically closed field $C$ of characteristic zero, containing
$C_0$. Unless otherwise noted, we take $C$ to have a transcendence
degree $\aleph_0$ over $C_0$.

Take $\K$ to be the category of models of the theory $T_\Ss^U$ which
have this given field $C$, with arrows being embeddings of
$\L_\Ss$-structures which fix $C$.
Because we are working in a more abstract setting than usual, the
following lemma actually requires a proof.
\begin{lemma}
  The category $\K$ has intersections, that is, for each $B \in \K$,
  and each family $(A_i \into B)_{i \in I}$ of substructures of $B$,
  there is a limit $\bigcap_{i \in I} A_i \into B$ of the obvious
  diagram this defines. Furthermore the underlying field of this
  intersection is simply the intersection of the underlying fields of
  the substructures.
\end{lemma}
\begin{proof}
  The axiomatization of $T_\Ss^U$ is universal, apart from the axiom
  scheme which says that the field is algebraically closed. The
  intersection of algebraically closed subfields of a field is
  algebraically closed, and any substructure of a model of a universal
  theory is also a model of that theory, so the category of models of
  $T_\Ss^U$ has intersections. The intersection of extensions of $C$
  is also an extension of $C$.
\end{proof}
Using this lemma, if $B \in \K$ and $X$ is a subset of $B$, we can
define the substructure of $B$ generated by $X$ as $\gen{X} = \bigcap
\class{A \into B}{X \subs A}$, where $A \into B$ means that $A$ is a
subobject of $B$ in $\K$. Note that $\gen{X}$ depends on $B$.

We say that $B$ is \emph{finitely generated} iff there is a finite
subset $X$ of $B$ such that $B = \gen{X}$. In fact, for any $A \in \K$
and subset $X$ of $A$, $\gen{X}$ is simply the algebraic closure of $C
\cup X$ in $A$, so an object $A$ of $\K$ is finitely generated iff
$\td(A/C)$ is finite. Thus being a finitely generated object of $\K$
is not the same as being finitely generated as an
$\L_\Ss$-structure. Indeed no objects of $\K$ are finitely generated
as $\L_\Ss$-structures since they are all algebraically closed
fields. \label{small in K}

We write $A \fgsub B$ to mean that $A$ is a finitely generated
substructure of $B$. From the above characterization it follows that
any substructure of a finitely generated structure in $\K$ is also
finitely generated.

\subsection{The predimension function}

The Schanuel property allows us to define a predimension function,
$\delta$, on the finitely generated objects of $\K$. It is defined
in terms of transcendence degree and a \emph{group rank}, which we
define using the next series of lemmas.

\begin{lemma}
  If $S_1$ and $S_2$ are isogenous then $\Gamma_{S_1}$ determines
  $\Gamma_{S_2}$.
\end{lemma}
\begin{proof}
  By the definition of isogeny, there are an $S_3$ and isogenies
  $f_1:S_3 \to S_1$ and $f_2: S_3 \to S_2$. By axiom U5, $\Gamma_{S_2}
  = (Tf_2)(Tf_1)^{-1}(\Gamma_{S_1})$.
\end{proof}

\begin{lemma}\label{Smax lemma}
  For any extension $A \into B$ in $\K$ with $B$ finitely generated,
  there is $S \in \Ss$ of maximal dimension such that there is $g \in
  \Gamma_S(B)$, not lying in an $A$-coset of $TH$ for any proper
  algebraic subgroup $H$ of $S$.  Furthermore, this maximal $S$ is
  uniquely defined up to isogeny, and determines $\Gamma$ on $B$ as
  follows. 

  If $g' \in \Gamma_{S'}(B)$ for any $S' \in \Ss$, then there is $S''$
  isogenous to $S$, $g'' \in \Gamma_{S''}(B)$, a homomorphism
  $\ra{S''}{q}{S'}$, and $\gamma \in \Gamma_{S'}(A)$, such that $g' =
  (Tq)(g'')\cdot \gamma$, where $\cdot$ is the group operation in
  $S'$.
\end{lemma}
\begin{proof}
  If $g \in \Gamma_S(B)$ and does not lie in an $A$-coset of $TH$ for
  any proper algebraic subgroup $H$ of $S$, then it does not lie in a
  $C$-coset and by the Schanuel property SP, $\dim S < \td(g/C)$ or
  $\dim S = 0$. Also $\td(g/C) \le \td(B/C)$, so the dimension of $S$
  is bounded. At least one such $S$ exists (the zero-dimensional
  group), and hence a maximal such $S$ exists.

  Now let $S$ be of maximal dimension and $g \in \Gamma_S(B)$ as
  described. Suppose $g' \in \Gamma_{S'}(B)$ for some $S' \in
  \Ss$. Then $(g,g') \in \Gamma_{S\cross S'}(B) \subs T(S\cross
  S')(B)$. By maximality of $\dim S$, there is an algebraic subgroup
  $S''$ of $S\cross S'$, with $\dim S'' \le \dim S$, such that
  $(g,g')$ lies in an $A$-coset of $TS''$. Let $(\alpha,\beta) \in
  \Gamma_{S \cross S'}(A)$ and $g'' \in \Gamma_{S''}(B)$ such that
  $(g,g') = g'' \cdot (\alpha,\beta)$. The projection maps
  \[
  \begin{diagram}[height=2em,width=2em]
    &&S \cross S'\\
    & \ldTo<{\pr_1} &&\rdTo>{\pr_2} \\
    S &&&&S'
  \end{diagram} \quad \mbox{restrict to} \quad
  \begin{diagram}[height=2em,width=2em]
    &&S''\\
    & \ldTo<{p} &&\rdTo>{q} \\
    S &&&&S'
  \end{diagram}\]
  and we also have the maps $Tp$, $Tq$ on the tangent bundles.
  Then $(Tp)(g'') = g \cdot (Tp)(\alpha)$, which lies in $T(p(S''))$,
  where $p(S'')$ is an algebraic subgroup of $S$. Now $g$ does not lie
  in $TH$ for any proper algebraic subgroup $H$ of $S$, so $p(S'') =
  S$. Hence $\dim S'' = \dim S$ and $p$ is an isogeny.
  Let $\gamma = (T\pr_2)(\beta)$. Then $g' = (Tq)(g'') \cdot \gamma$,
  where $g'' \in \Gamma_{S''}(B)$ and $\gamma \in \Gamma_{S'}(A)$ as
  required. 

  If $\dim S' = \dim S$ then the same argument shows that $q$ is an
  isogeny. Hence $S$ is unique up to isogeny.
\end{proof}

\begin{defn}\label{Smax defn}
  For an extension $A \into B$ in $\K$, with $B$ finitely generated,
  define $\Smax(B/A)$ to be a maximal $S \in \Ss$ such that there is
  $g \in \Gamma_S(B)$, not lying in an $A$-coset of $TH$ for any
  proper algebraic subgroup $H$ of $S$. A point $g \in
  \Gamma_{\Smax(B/A)}$ which witnesses the maximality is said to be a
  \emph{basis for $\Gamma(B/A)$}. For a finitely generated $A \in \K$,
  define $\Smax(A) = \Smax(A/C)$.

  Note that $\Smax(B/A)$ is defined only up to isogeny.
\end{defn}

\begin{prop}\label{Smax extensions}
  Let $A, B \in \K$ be finitely generated, with $B$ an extension of
  $A$, that is, $A \subs B$. Then $\Smax(B)$ is an extension of
  $\Smax(A)$ in the group theory sense, that is, $\Smax(A)$ is a
  quotient of $\Smax(B)$. Furthermore, $\Smax(B/A)$ is the kernel of
  the quotient map.
\end{prop}
\begin{proof}
  Let $b \in \Gamma_{\Smax(B)}$ and $a \in \Gamma_{\Smax(A)}$ be
  bases, and write $S_B$ for $\Smax(B)$ and $S_A$ for
  $\Smax(A)$. Then, replacing $S_B$ by an isogenous group if
  necessary, there is a quotient map $\ra{S_B}{q}{S_A}$, and $\gamma
  \in S_A(C)$ such that $a = (Tq)(b)\cdot \gamma$. Thus $(Tq)(b) \in
  S_A(A)$, so $b$ lies in an $A$-coset of $TH$, where $H$ is the
  kernel of $q$. Say $b = e \cdot \alpha$, with $e \in H(B)$ and
  $\alpha \in S_B(A)$.

  We will show that $e$ is a basis for $\Gamma(B/A)$. Firstly, $e$
  does not lie in an $A$-coset of $TJ$ for any proper algebraic
  subgroup $J$ of $H$, since then $S/J$ would be $\Smax(A)$. If $g \in
  \Gamma_S(B)$ then, by the properties of $\Smax(B)$, up to isogeny
  there is $\ra{S_B}{p}{S}$ such that $g = (Tp)(b)\cdot\beta$, for
  some $\beta \in S(A)$. But then $g = (Tp)(e)\cdot
  (Tp)(\alpha)\cdot\beta$, and $(Tp)(\alpha)\cdot\beta \in
  S(A)$. Hence $e$ is a basis for $\Gamma(B/A)$, and $H = \Smax(B/A)$.
\end{proof}

\begin{defn}\label{grk delta defn}
  For an extension $A \into B$ in $\K$, with $B$ finitely generated,
  define the \emph{group rank} and \emph{predimension} to be
  \[\grk(B/A) = \dim \Smax(B/A) \qquad \qquad \delta(B/A) = \td(B/A) -
  \grk(B/A)\] respectively. For any subset $X \subs B$, define
  $\grk(X/A) = \grk(\gen{X,A}/A)$ and define $\delta(X/A) =
  \delta(\gen{X,A}/A)$.
  Also define $\grk(A) = \grk(A/C)$ and $\delta(A) = \delta(A/C)$.
\end{defn}
The Schanuel property says precisely that $\delta(A) \ge 0$ for each
finitely generated structure $A$, with equality iff $A = C$.

\begin{lemma}\label{relative delta}
  For an extension $A \into B$ in $\K$, with $B$ finitely generated,
  $\grk(B) = \grk(B/A) + \grk(A)$ and $\delta(B) = \delta(B/A) +
  \delta(A)$.
\end{lemma}
\begin{proof}
  The statement for group rank is immediate from proposition~\ref{Smax
  extensions}. The same property holds for transcendence degree, that
  is, $\td(B/C) = \td(B/A) + \td(A/C)$, and the result for the
  predimension follows.
\end{proof}

An essential property of $\delta$ is that it is \emph{submodular}.
\begin{lemma}\label{submodularity}
  The predimension $\delta$ is submodular on $\K$. That is, for any
  finitely generated $B \in \K$ and any $A_1, A_2 \subs B$, such that
  $A_1,A_2 \in \K$,
  \[\delta(A_1 \cup A_2) + \delta(A_1 \cap A_2) \le \delta(A_1) +
  \delta(A_2).\]
\end{lemma}
\begin{proof}
  Let $A_0 = A_1 \cap A_2$, and $A_3 = \gen{A_1 \cup A_2}$. We first
  show that $\grk(A_3/A_0) \ge \grk(A_1/A_0) + \grk(A_2/A_0)$. For
  $i=1,2,3$, let $S_i = \Smax(A_i/A_0)$ and let $g_i \in TS_i$ be a
  basis for $\Gamma(A_i/A_0)$. By lemma~\ref{Smax lemma}, there are
  (up to isogeny) homomorphisms $\ra{S_3}{q_i}{S_i}$ for $i=1,2$ such
  that $g_i = (Tq_i)(g_3)$. 

  Suppose $\grk(A_3/A_0) < \grk(A_1/A_0) + \grk(A_2/A_0)$. Then $\dim
  S_3 < \dim S_1 \cross S_2$, and by definition of $\Smax$, there is a
  proper algebraic subgroup $H$ of $S_1 \cross S_2$ such that
  $(g_1,g_2)$ lies in an $A_0$-coset of $TH$. Now $H$ is normal in
  $S_1 \cross S_2$ since the groups are commutative, so it is the
  kernel of some algebraic group homomorphism $\ra{S_1\cross
    S_2}{p}{J}$, and $(Tp)(g_1,g_2) = \alpha \in TJ(A_0)$. Since the
  product $S_1\cross S_2$ is also the direct sum $S_1 \oplus S_2$ in
  the category of commutative algebraic groups, we can write
  $(Tp)(g_1,g_2)$ as $(Tp_1)(g_1)\cdot (Tp_2)(g_2)$, where
  $\ra{S_1}{p_1}{J}$ is given by $p_1(x) = p(x,1)$, and symmetrically
  $p_2$. Then $(Tp_1)(g_1) = \alpha \cdot (Tp_2)(g_2)^{-1} \in
  TJ(A_2)$, so $Tp_1(g_1) \in TJ(A_1 \cap A_2) = TJ(A_0)$. Thus $g_1$
  lies in an $A_0$-coset of $T(\ker p_1)$, but $\dim J > 0$, so $\ker
  p_1$ is a proper algebraic subgroup of $S_1$, which contradicts
  $g_1$ being a basis of $\Gamma(A_1/A_0)$. So $\grk(A_3/A_0) \ge
  \grk(A_1/A_0) + \grk(A_2/A_0)$. Thus
  \[\grk(A_1\cup A_2/A_0) + 2\grk(A_0) \ge \grk(A_1/A_0) +
  \grk(A_2/A_0) + 2\grk(A_0)\] 
  and hence, by lemma~\ref{relative delta},
  \begin{equation}\label{grkeqn}
    \grk(A_1\cup A_2) + \grk(A_1 \cap A_2) \ge \grk(A_1) + \grk(A_2).
  \end{equation}
  Now
  \begin{equation}\label{tdeqn}
    \td(A_1 \cup A_2/C) + \td(A_1 \cap A_2/C) \le \td(A_1/C) +
    \td(A_2/C)
  \end{equation}
  and so, subtracting (\ref{grkeqn}) from (\ref{tdeqn}), we see that
  $\delta$ is submodular.
\end{proof}

\subsection{Self-sufficient embeddings}

The intuition behind the predimension function $\delta$ is that is
measures the number of ``degrees of freedom'', which could be thought
of as the number of variables minus the number of constraints. We
cannot amalgamate over all embeddings because an amalgam of arbitrary
embeddings will not always have the Schanuel property. That is, $\K$
does not have the amalgamation property. The problem is that for some
embeddings $A \into B$ there will be extra constraints on $A$ which
are not apparent in $A$ but are witnessed only in the extension $B$.
We will amalgamate only over those embeddings where this does not
occur. Informally, an embedding $A \into B$ is \emph{self-sufficient}
if any dependency (constraint) between members of $A$ in $B$ is
already witnessed in $A$. The formal definition does not require the
structures to be finitely generated.
\begin{defn}
  We say that an embedding of structures $A \into B$ is
  \emph{self-sufficient} iff for every $X \fgsub B$ we have $\delta(X
  \cap A) \le \delta(X)$. In this case, we write the embedding as $A
  \strong B$ or $A \rStrong B$ and we say that $A$ is self-sufficient
  in $B$.
\end{defn}

\begin{lemma}
  Taking all the objects of $\K$ with just the self-sufficient
  embeddings gives a subcategory $\Ks$ of $\K$.
\end{lemma}
\begin{proof}
  It is immediate that identity embeddings are self-sufficient and the
  composite of self-sufficient embeddings is self-sufficient.
\end{proof}

It is customary to write self-sufficient embeddings as $A \le B$, but
this seems to me to be an unnecessary duplication of a common symbol
and potentially confusing, so I prefer to avoid it. This is a
simplification of the original definition of a self-sufficient
embedding (see for example \cite{Hru93}), and it is equivalent to the
original definition for any $\delta$ which is submodular, as
predimension functions for Hrushovski-type constructions are.

\begin{lemma}\label{Ks has intersections}
  If $A_i \strong B$ for each $i$ in some index set $I$ and $A =
  \bigcap_{i \in I}A_i$ is the intersection in $\K$, then $A \strong
  B$. In particular, the category $\Ks$ has intersections.
\end{lemma}
\begin{proof}
  First we show that it holds for binary intersections. Suppose
  $A_1,A_2 \strong B$. Let $X \fgsub A_1$. Then $\delta(X \cap (A_1
  \cap A_2)) = \delta(X \cap A_2) \le \delta(X)$ since $A_2 \strong B$
  and $X \fgsub B$. So $A_1 \cap A_2 \strong A_1$, but also $A_1
  \strong B$ and so $A_1 \cap A_2 \strong B$. By induction, any finite
  intersection of self-sufficient substructures of $B$ is also
  self-sufficient in $B$.

  The case of an arbitrary intersection of self-sufficient subsets
  follows by a finite character argument. Let $X \fgsub B$. Then $X
  \cap \bigcap_{i \in I}A_i$ is an algebraically closed subfield of
  $X$, which has finite transcendence degree. The lattice of
  algebraically closed subfields of $X$ has no infinite chains, hence
  there is a finite subset $I_0$ of $I$ such that $X \cap \bigcap_{i
  \in I}A_i = X \cap \bigcap_{i \in I_0}A_i$. By the above,
  $\bigcap_{i \in I_0}A_i \strong B$, and so $\delta(X \cap \bigcap_{i
  \in I}A_i) \le \delta(X)$. So $\bigcap_{i \in I}A_i \strong B$ as
  required.
\end{proof}

As with $\K$, the existence of intersections allows one to define the
subobject \emph{generated} by some set, and consequently the notion of
a finitely generated object in $\Ks$. This greatly simplifies the
presentation, and is one reason for working in the category $\K$
rather than the category of all $\L_\Ss$-substructures of models. To
distinguish this notion of generation from that in $\K$, we give it a
different name.
\begin{defn}
  If $B$ is a structure and $X$ is a subset of $B$ then the
  \emph{hull} of $X$ in $B$ is given by $\hull{X} = \bigcap \class{A
  \strong B}{X \subs A}$.
\end{defn}
Note that as for $\gen{X}$, the hull $\hull{X}$ depends on $B$,
although we do not write the dependence explicitly. Hulls give another
way of showing that an embedding is self-sufficient.
\begin{lemma}\label{hulls for selfsuff}
  $A \strong B$ iff for every $Y \fgsub A$, $\hull{Y} \subs A$.
\end{lemma}
\begin{proof}
  Suppose $Y \fgsub A$ and $\hull{Y} \not\subs A$. Let $X =
  \hull{Y}$. Then $\delta(X) < \delta(X \cap A)$, so $A \nstrong
  B$. Conversely, suppose $A \nstrong B$, so there is $X \fgsub B$
  such that $\delta(X) < \delta(X \cap A)$. Then $X \cap A$ is
  finitely generated so take $Y = X \cap A$.
\end{proof}

\subsection{The amalgamation property}

\begin{lemma}\label{finitely generated}
A structure is finitely generated in the sense of $\Ks$ iff it is
finitely generated in the sense of $\K$.
\end{lemma}
\begin{proof}
The right to left direction is immediate, since for any set $X$,
$\gen{X} \subs \hull{X}$.

We show that if $B \in \K$ and $X \subs B$ is a finite subset then
$\hull{X}$ is finitely generated in $\K$. Consider
$\class{\delta(A)}{X \subs A \fgsub B}$, a nonempty subset of
$\N$. Let $A$ be such that $\delta(A)$ is least. Then for any $Y
\fgsub B$,
\[0 \le \delta(A \cup Y) - \delta(A) \le \delta(Y) - \delta(A \cap
Y)\]
with the first comparison holding by the minimality of $\delta(A)$ and
the second by submodularity of $\delta$. Thus $A \strong B$. In
particular, $\hull{X} \subs A$, and so $\hull{X}$ is finitely
generated in $\K$.
\end{proof}

We define the category $\Kso$ to be the subcategory of $\Ks$ consisting
of the finitely generated structures, together with all
self-sufficient embeddings. In order to apply the amalgamation
theorem, we need to show that $\Kso$ has the amalgamation property. In
fact, we show more than this, which is necessary when it comes to
axiomatizing the amalgam.
\begin{prop}[Free asymmetric amalgamation]\label{free amalgamation}
  If we have embeddings $A \strong B_1$ and $A \into B_2$ in $\K$
  then there is $E \in \K$ (the \emph{free amalgam} of $B_1$ and
  $B_2$ over $A$) and embeddings $B_1 \into E$ and $B_2 \strong E$
  such that the square
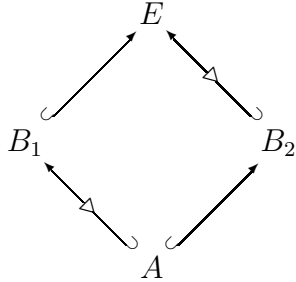
\begin{diagram}[height=2em,width=2em]
&&E\\
&\ruInto&&\luStrong\\
B_1&&&&B_2\\
&\luStrong&&\ruInto\\
&&A
\end{diagram}
commutes, and $E = \gen{B_1,B_2}$. Furthermore, if $A \strong B_2$
then $B_1 \strong E$.
\end{prop}
\begin{proof}
  Let $\beta_1, \beta_2$ be transcendence bases of $B_1, B_2$ over
  $A$. As a field, take $E$ to be the algebraic closure of the
  extension of $A$ with transcendence base the disjoint union $\beta_1
  \sqcup \beta_2$. This defines the field $E$ and the embeddings $B_1
  \into E$ and $B_2 \into E$ uniquely up to isomorphism, because $B_1$
  and $B_2$ are algebraically disjoint over $A$ in $E$. For each $S
  \in \Ss$, define $\Gamma_S(E)$ to be the subgroup of $TS(E)$
  generated by $\Gamma_S(B_1) \cup \Gamma_S (B_2)$. Axioms U1---U7
  then hold by the construction.

  Let $X$ be a finitely generated algebraically closed substructure of
  $E$. Note that $\delta$ and $\grk$ were originally defined only for
  structures satisfying the Schanuel property, and we do not yet know
  that it holds for $E$. However, the definitions of $\delta$ and
  $\grk$ make sense for $X$ because the conclusion of lemma~\ref{Smax
  lemma} holds, and so $\grk(X)$ is well-defined and finite.

  Let $S = \Smax(X/X \cap B_2)$, and let $g\in \Gamma_S(X)$ be a basis
  for $\Gamma(X/X \cap B_2)$. Then by the construction of
  $\Gamma_S(E)$, there are $h \in \Gamma_S(B_1)$ and $b \in
  \Gamma_S(B_2)$ such that $g = h\cdot b$. The group operation of $S$
  is defined over $C$, so certainly over $B_2$, and so
  \[\td(g/X \cap B_2) \ge \td(g/B_2) = \td(h/B_2) = \td(h/A) \ge \grk(h/A)\]
  with the second equation because $B_1$ is algebraically independent
  of $B_2$ over $A$ and the final comparison because $A \strong B_1$.

  We now show that $\grk(h/A) = \dim S$. If not, then there is $a \in
  TS(A)$ and a proper algebraic subgroup $H$ of $S$ such that $h\cdot
  a^{-1} \in TH(B_1)$. Now $h\cdot a^{-1} = g\cdot (a\cdot b)^{-1}$,
  and $a\cdot b \in TS(B_2)$, so $g$ lies in a $B_2$-coset of
  $TH$. This contradicts the fact that $g$ is a basis for
  $\Gamma(X/X \cap B_2)$. So $\grk(h/A) = \dim S$, and thus
  \[\delta(X/X \cap B_2) \ge \td(g/X \cap B_2) - \dim S \ge 0.\]
  Thus $B_2 \strong E$. The symmetric argument shows that if $A
  \strong B_2$ then $B_1 \strong E$.

  Now $\delta(X \cap B_2) \ge 0$ because $B_2$ satisfies SP, so 
  \[\delta(X) = \delta(X/X\cap B_2) + \delta(X \cap B_2) \ge 0.\]
  Suppose $\delta(X) = 0$. Let $e \in \Gamma_{S'}$ be a basis for
  $\Gamma(X\cap B_2/C)$. Then
  \[0 = \delta(X) = \td(g/C(e)) + \td(e/C) - \dim S - \dim S'\] but,
  since $B_2$ satisfies SP, either $\td(e/C) > \dim S'$ or $X \cap B_2
  \subs C$. By the calculation above, $\td(g/C(e)) \ge \td(g/X \cap
  B_2) \ge \dim S$, so $\td(e/C) \le \dim S'$, and hence $X \cap B_2
  \subs C$. Thus $g$ is independent from $B_2$ over $C$, so $g \in
  B_1$. But then $\td(g/C) = \dim S$, so $X \subs C$ using SP for
  $B_1$. Hence $E$ has SP.
\end{proof}

\subsection{The amalgamation theorem}

 The category $\Ks$ is not the category of all finitely generated models of a universal first order theory, because its objects are all algebraically closed field extensions of a fixed $C$. Thus we must use a more abstract version of the Fraiss\'e amalgamation theorem than that given, for example, in \cite{Hodges93}. We use a variant of the category-theoretic version given in \cite{DG92}. We must explain how some standard notions are translated into this setting.

Fix an ordinal $\lambda$, and consider a category $\Cat$.  A
\emph{chain} of length $\lambda$ in $\Cat$ is a collection $(Z_i)_{i <
\lambda}$ of objects of $\Cat$ together with arrows
$\ra{Z_i}{\gamma_{ij}}{Z_j}$ for each $i \le j < \lambda$, such that
for each $i$, $\lambda_{ii} = 1_{Z_i}$, and if $i \le j \le k <
\lambda$ then $\gamma_{jk} \circ \gamma_{ij} = \gamma_{ik}$. The
\emph{union} or \emph{direct limit} of a $\lambda$-chain is an object
$Z = Z_\lambda$ with arrows $\ra{Z_i}{\gamma_{i\lambda}}{Z}$ for each
$i < \lambda$, satisfying the usual universal property of a direct
limit.

For $\lambda$ an infinite regular cardinal, identified with its
initial ordinal, an object $A$ of $\Cat$ is said to be
\emph{$\lambda$-small} iff for every $\lambda$-chain $(Z_i,
\gamma_{ij})$ in $\Cat$ with direct limit $Z$, any arrow
$\ra{A}{f}{Z}$ factors through the chain, that is, there is $i <
\lambda$ and $\ra{A}{f^*}{Z_i}$ such that $f = \gamma_{i\lambda} \circ
f^*$. For example, in the category of sets a set is $\aleph_0$-small
iff it is finite.  Write $\Cat_{<\lambda}$ for the full subcategory of
$\Cat$ consisting of all the $\lambda$-small objects of $\Cat$, and
$\Cat_{\le \lambda}$ for the full subcategory of $\Cat$ consisting of
all unions of $\lambda$-chains of $\lambda$-small objects.

\begin{defn}
  We say that $\Cat$ is a \emph{$\lambda$-amalgamation category} iff the
  following hold.
  \begin{itemize}
  \item Every arrow in $\Cat$ is a monomorphism.
  \item $\Cat$ has direct limits (unions) of chains of every ordinal
    length up to $\lambda$.
  \item $\Cat_{<\lambda}$ has at most $\lambda$ objects up to
    isomorphism.
  \item For each object $A \in \Cat_{<\lambda}$ there are at most
    $\lambda$ extensions of $A$ in $\Cat_{<\lambda}$, up to isomorphism.
  \item $\Cat_{<\lambda}$ has the \emph{amalgamation property} (AP),
    that is, any diagram of the form
    
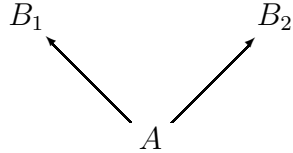
\begin{diagram}[height=2em,width=2em]
      B_1 &&&&B_2\\
      &\luTo&& \ruTo\\
      &&A
    \end{diagram}
    can be completed to a commuting square
    
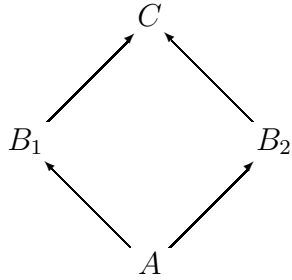
\begin{diagram}[height=2em,width=2em]
      &&C\\
      &\ruTo&&\luTo\\
      B_1 &&&&B_2\\
      &\luTo&& \ruTo\\
      &&A
    \end{diagram}
    in $\Cat_{<\lambda}$.
  \item $\Cat_{<\lambda}$ has the \emph{joint embedding property} (JEP),
    that is, for every $B_1,B_2 \in \Cat_{<\lambda}$ there is $C \in
    \Cat_{<\lambda}$ and arrows
    
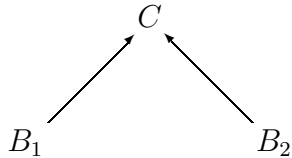
\begin{diagram}[height=2em,width=2em]
      &&C\\
      &\ruTo&&\luTo\\
      B_1 &&&&B_2\\
    \end{diagram}
    in $\Cat_{<\lambda}$.
  \end{itemize}
\end{defn}

An \emph{extension} of $A$ is simply an arrow with domain $A$. To say
that two extensions $\ra{A}{f}{B}$ and $\ra{A}{f'}{B'}$ are isomorphic
means that there is an isomorphism $\ra{B}{g}{B'}$ such that $f'=gf$.
In \cite{DG92}, Droste and G\"obel consider a stronger condition than
bounding the number of extensions of each $A$, namely that for any
pair of objects $A$ and $B$ there are only $\lambda$ arrows from $A$
to $B$. This allows them to use the pre-existing definition of a
$\lambda$-algebroidal category, but it is not strong enough for our
purposes. For example, if $A$ is a pure algebraically closed field
extension of $C$ of transcendence degree one then there are
$2^{\aleph_0}$ embeddings of $A$ into itself over $C$, but they are
all isomorphisms, and hence isomorphic extensions. The condition
bounding only the number of extensions is model-theoretically much
more natural.

To say that an object $U$ of $\Cat$ is $\Cat_{\le\lambda}$-universal means
that for every object $A \in \Cat_{\le\lambda}$ there is an arrow
$\ra{A}{}{U}$ in $\Cat$.  To say that $U$ is $\Cat_{<\lambda}$-saturated
means that for any $A,B \in \Cat_{<\lambda}$ and any arrows
$\ra{A}{f}{U}$ and $\ra{A}{g}{B}$ there is an arrow $\ra{B}{h}{U}$
such that $h \circ g = f$. These are just the translations into
category-theoretic language of the usual model-theoretic notions.

\begin{theorem}[Amalgamation theorem]\label{amalgamation theorem}
  If $\Cat$ is a $\lambda$-amalgamation category then there is an object
  $U \in \Cat_{\le\lambda}$, the ``Fraiss\'e limit'', which is $\Cat_{\le\lambda}$-universal and
  $\Cat_{<\lambda}$-saturated. Furthermore, $U$ is unique up to
  isomorphism.
\end{theorem}
\begin{proof}
  The proof in \cite{DG92} goes through, even with the slightly weaker
  hypothesis bounding the number of extensions rather than the number
  of arrows.
\end{proof}

The notion of $\aleph_0$-small is the same as finitely generated in
our example.
\begin{lemma}\label{aleph_0 small}
  An object $A$ of $\K$ is $\aleph_0$-small in $\K$ or in $\Ks$ iff it
  is finitely generated (that is, iff $\td(A/C)$ is finite).
\end{lemma}
\begin{proof}
  If $A$ is finitely generated by $x_1,\ldots,x_n$ and $A \into Z$
  where $Z$ is the union of an $\omega$-chain $(Z_i)_{i <\omega}$ then
  each $x_j$ lies in some $Z_{i(j)}$, so taking $i$ greater than each
  $i(j)$ the embedding factors through $Z_i$. This argument works for
  both categories $\K$ and $\Ks$.

  Conversely, if $\td(A/C)$ is infinite, let $X \cup \{x_j\}_{j <
  \omega}$ be an transcendence base for $A$ over $C$, and let $Z_i =
  \gen{X \cup \class{x_j}{j \le i}}$. Then $A$ is the union of the
  chain $(Z_i)$ in $\K$, but is not equal to any of the $Z_i$. Hence
  it is not $\aleph_0$-small in $\K$.

  Now let $W_i = \hull{X \cup \class{x_j}{j \le i}}$. By
  lemma~\ref{finitely generated} together with the existence of free
  amalgams, $\td(\hull{B}/B)$ is finite for any $B$. Thus $W_i$ is an
  $\omega$-chain in $\Ks$, with a strictly increasing cofinal subchain
  and union $A$, and so $A$ is not $\aleph_0$-small in $\Ks$.
\end{proof}

\begin{lemma}\label{finitely generated extensions}
  Let $A \in \K$ and let $B$ be a self-sufficient extension of $A$
  which is finitely generated over $A$. Then $B$ is determined up to
  isomorphism by $\Smax(B/A)$, the algebraic locus $\loc_A(g)$ of a
  basis $g$ for $\Gamma(B/A)$, and the natural number $\td(B/A(g))$.
\end{lemma}
\begin{proof}
  As a field extension, $B$ is determined by its transcendence degree
  over $A$. By lemma~\ref{Smax lemma}, the points of $\Gamma_S(B)$ for
  each $S \in \Ss$ are determined by $\Smax(B/A)$ and the basis $g$.
\end{proof}

\begin{prop}\label{Ks is amalg cat}
$\Ks$ is an $\aleph_0$-amalgamation category.
\end{prop}
\begin{proof}
  Every embedding in $\Ks$ is certainly a monomorphism, because $\Ks$
  is a concrete category and the underlying function is injective. It
  is also easy to see that $\Ks$ has unions of chains of any ordinal
  length, and in particular unions of $\omega$-chains.

  There are only countably many $S \in \Ss$, and only countably many
  algebraic varieties defined over $A$, so by lemma~\ref{finitely
  generated extensions} there are only countably many self-sufficient
  extensions of $A$. The structure $C$ embeds self-sufficiently into
  every $B \in \Ks$, so taking $A = C$ it follows in particular that
  $\Kso$ has only countably many objects.  The amalgamation property
  for $\Kso$ is given by proposition~\ref{free amalgamation}, and the
  joint embedding property follows from the amalgamation property,
  again since $C$ embeds self-sufficiently into each $B \in \K$.
\end{proof}

Putting proposition~\ref{Ks is amalg cat} and
theorem~\ref{amalgamation theorem} together, we get the universal
structure we want.
\begin{theorem}\label{amalgam is unique}
  There is a countable model $U$ of $T_\Ss^U$ which is universal and
  saturated with respect to self-sufficient embeddings. Furthermore,
  $U$ is unique up to isomorphism.\qed
\end{theorem}

Note that this Fraiss\'e limit $U$ is a union of an $\omega$-chain of countable
structures, hence is countable. Every countable model of $T_\Ss^U$ can
be self-sufficiently embedded in some $A \in \K_{\le\aleph_0}$, by
extending the constant field and taking the algebraic closure. Thus
the $\Ks_{\le\aleph_0}$-universality of $U$ implies that every
countable model of $T_\Ss^U$ can be self-sufficiently embedded into
$U$. Similarly, $U$ is saturated with respect to self-sufficient
embeddings for any self-sufficient substructures of finite
transcendence degree.

\subsection{Pregeometry and dimension}\label{pregeometry section}

The geometry of the Fraiss\'e limit $U$ is controlled by a pregeometry, which we
now describe. For any model $M$ of $T_\Ss^U$, in particular $U$, the
predimension function $\delta$ gives rise to a \emph{dimension} notion
on $M$. The dimension function is conventially denoted $d$ and is
defined as follows.

\begin{defn}
  For $X \finsub M$ (or even $X \subs M$ with $\td(X/C)$ finite),
  define $d(X) = \delta(\hull{X})$ or, equivalently, $d(X)= \min
  \class{\delta(XY)}{Y \finsub M}$.\\
  For $X$ as above and any $A \subs M$, the dimension of $X$ over $A$
  is defined to be
  \[d(X/A) = \min \class{d(XY) - d(Y)}{Y \finsub A}.\]
\end{defn}
Note that $d(X) = d(X/\emptyset)$, so the two definitions agree.

\begin{lemma}[Properties of d]\label{Properties of d}
  Let $X,Y \finsub M$ and $A,B \subs M$.
  \begin{enumerate}
  \item If $X \subs Y$ then $d(X/A) \le d(Y/A)$.
  \item If $A \subs B$ then $d(X/A) \ge d(X/B)$.
  \item $d$ is submodular: $d(XY) + d(X \cap Y) \le d(X) + d(Y)$.
  \item $d(X/Y) = d(XY) - d(Y)$.
  \item $d(X) \ge 0$, with equality iff $X \subs C$.
  \item For any $x \in M$, $d(x/A) = 0$ or $1$.
  \end{enumerate}
\end{lemma}
\begin{proof}
The first two parts are immediate from the definition.
For submodularity:
\begin{eqnarray*}
  d(XY) + d(X\cap Y) &=& \delta(\hull{XY}) + \delta(\hull{X\cap Y})\\
  &\le& \delta(\hull{X}\hull{Y}) + \delta(\hull{X} \cap \hull{Y})\\
  &\le& \delta(\hull{X}) + \delta(\hull{Y})\\
&& = d(X) + d(Y)
\end{eqnarray*}
For part 4, let $Z\subs Y$. Then 
\[d(XY) - d(Y) \le d(XZ) - d(XZ \cap Y) \le d(XZ) - d(Z)\]
by submodularity and monotonicity of $d$. Thus the minimum
value of $d(XZ) - d(Z)$ occurs when $Z=Y$.

Part 5 follows from the Schanuel property.

For part 6, take $A_0 \finsub A$ such that $d(x/A) = d(x/A_0)$. Then
\begin{eqnarray*}
d(x/A_0) & = & d(A_0x) - d(A_0)\\
& = & \delta(\hull{A_0x}) - \delta(\hull{A_0})\\ 
& = & \delta(\hull{\hull{A_0}x}) - \delta(\hull{A_0})\\
& \le & \delta(\hull{A_0}x) - \delta(\hull{A_0})\\
& \le & \td(x/\hull{A_0}) \le 1
\end{eqnarray*} 
so $d(x/A_0) = 0$ or $1$.
\end{proof}

\begin{prop}\label{cl is pregeometry}
  The operator $\ra{\powerset M}{\cl}{\powerset M}$ given by $x \in
  \cl{A} \iff d(x/A) = 0$ is a pregeometry on $M$. If $X \subs M$ is
  such that $d(X)$ is defined (that is, $\td(X/C)$ is finite) then
  $d(X)$ is equal to the dimension of $X$ in the sense of the
  pregeometry.
\end{prop}
\begin{proof}
  It is straightforward to check that $\cl$ is a closure operator with
  finite character. It remains to check the exchange property. Let $A
  \subs M, a,b \in M$, and $a \in \cl(Ab)\minus \cl(A)$.  By finite
  character, there is a finite $A_0 \subs A$ such that $a \in
  \cl(A_0b)$. Then $d(a/A_0) = 1$.
Using part 4 of lemma~\ref{Properties of d}, we have 
\begin{eqnarray*}
d(b/A_0a) &=& d(A_0ab) - d(A_0a) \\
 & = & d(A_0b) - d(A_0a)\\
 & = & [d(A_0) + d(b/A_0)] - [d(A_0) + d(a/A_0)]\\
 & = & [d(A_0) + 1] - [d(A_0) + 1] = 0
\end{eqnarray*}
and so $b \in \cl(Aa)$.

Finally, $x$ is independent from $A$ iff $d(x/A) = 1$, and so $d$
agrees with the dimension coming from the pregeometry.
\end{proof}

From now on, by the \emph{dimension} of a structure $A \in \K$ we mean
the dimension in the sense of this pregeometry on $A$. Note that
self-sufficient embeddings are precisely those embeddings which
preserve the dimension.

\subsection{Freeness and Rotundity}\label{rotundity section}

To explain what the theory of the structure $U$ is, we must translate
$\Kso$-saturation into a more tractable form. We will show that it is
equivalent to saying that certain algebraic subvarieties of $TS$ have
a nonempty intersection with $\Gamma_S$.

\begin{defn}\label{rotund defn}
  An irreducible subvariety $V$ of $TS$ is \emph{free} iff $V$ is not
  contained in a coset of $TH$ for any proper algebraic subgroup $H$
  of $S$. It is \emph{absolutely free} iff $\pr_S V$ is not contained
  in a coset of any such $H$ and $\pr_{LS} V$ is not contained in a coset
  of $LH$ for any such $H$.

  A point $g \in TS$ is (\emph{absolutely}) \emph{free} over a field
  $A$ iff $\loc_A(g)$ is (absolutely) free.

  An irreducible subvariety $V$ of $TS$ is \emph{rotund} iff
  for every quotient map $S \rOnto^f H$,
  \[\dim (Tf)(V) \ge \dim H\]
  and \emph{strongly rotund} iff for every such $f$ with $H \neq 1$,
  \[\dim (Tf)(W) \ge \dim H + 1.\]

  A point $g \in TS$ is (\emph{strongly}) \emph{rotund} over a field
  $A$ iff $\loc_A(g)$ is (strongly) rotund. A reducible variety is
  (\emph{strongly}) \emph{rotund} iff at least one of its irreducible
  components is.
\end{defn}

\begin{lemma}
  Let $A \strong B$ be a self-sufficient extension in $\Ks$, with $B$ finitely generated over $A$. Let $S = \Smax(B/A)$ and let $g \in \Gamma_S$ be a basis for $B$ over $A$. 

Then $g$ is free over $A$, absolutely free over $C$, rotund over $A$, and strongly rotund over $C$.
\end{lemma}
\begin{proof}
  If $S = 1$ then the result is trivial. Assume $S \neq 1$. 
  By the definition of a basis, $g$ does not lie in an $A$-coset of
  $TH$ for any proper algebraic subgroup $H$ of $S$, and hence
  $\loc_A(g)$ is not contained in such a coset. By axioms U4 and U5,
  $\pr_{LS}\loc_C(g)$ lies in a coset of $LH$ iff $\pr_S\loc_C(g)$
  lies in a coset of $H$, since $g \in \Gamma_S$. If both held then
  $g$ would lie in a $C$-coset of $TH$, but it does not, so $g$ is
  absolutely free over $C$.

  For each quotient map $S \rOnto^f H$,
  \[\dim((Tf)(\loc_A g)) - \dim H = \dim(\loc_A((Tf)(g)) - \dim H =
  \delta((Tf)(g)/A) \ge 0\] 
  as $g$ is free over $A$ and $A \strong B$, so $g$ is rotund over $A$.

  Similarly, if $H \neq 1$ then
  \[\dim((Tf)(\loc_C g)) - \dim H = \dim(\loc_C((Tf)(g)) - \dim H =
  \delta((Tf)(g)/C) \ge 1\] 
  as $B$ satisfies the Schanuel property, so $g$ is strongly rotund over $C$.
\end{proof}

It is useful to isolate the subvarieties which occur as the locus of a
basis of $\Gamma(B/A)$ for an extension $B/A$ which cannot be split
into a tower of smaller extensions. We call these \emph{perfectly
  rotund} subvarieties.
\begin{defn}
  A subvariety $V$ of $TS$ is \emph{perfectly rotund} iff it is rotund,
  $\dim V = \dim S$, and for every proper, nontrivial quotient map $S
  \rOnto^f H$,
  \[\dim (Tf)(V) > \dim H.\]
\end{defn}

\subsection{Existential closedness}\label{EC section}

\begin{defn}\label{parametric family}
  Let $X$ be a variety. Any constructible set $P$ and Zariski-closed $V \subs X \cross P$ defines a \emph{parametric family} $(V_p)_{p \in P}$ of subvarieties of $X$, where $V_p$ is the fibre above $p$ of the natural projection $X \cross P \to P$, restricted to $V$. We write $(V_p)_{p \in P(C)}$ to be the fibres over the $C$-points of $P$, and also call this a parametric family.
\end{defn}

\begin{defn}
  We consider three forms of Existential Closedness, and two notions
  relating to dimension: Non-Triviality and Infinite Dimensionality,
  for a model $M$ of $T_\Ss^U$.
  \begin{description}\label{EC defn}
  \item[EC] For each $S \in \Ss$, each irreducible rotund subvariety $V$ of
    $TS$, and each parametric family $(W_e)_{e \in Q(C)}$ of proper subvarieties of $V$, with $Q$ a constructible set defined over $C_0$, there is $g \in \Gamma_S \cap V \minus \bigcup_{e \in Q(C)} W_e$.

  \item[EC$'$] The same as EC except only for perfectly rotund $V$.
  \item[SEC (Strong existential closedness)] For each $S \in \Ss$,
    each rotund subvariety $V$ of $TS$, and each finitely generated
    field of definition $A$ of $V$, the intersection $\Gamma_S \cap V$
    contains a point which is generic in $V$ over $A \cup C$.
  \item[NT] There is $x \in M$ such that $x \notin C$.
  \item[ID] The structure $M$ is infinite dimensional.
  \end{description}
\end{defn}
NT is equivalent to saying that the dimension of $M$ is nonzero, so it
is implied by ID. Clearly SEC implies EC and EC implies EC$'$. We prove
that EC$'$ implies EC using of the tool of intersecting a
variety with generic hyperplanes. For $p = (p_1,\ldots,p_N) \in \mathbb{A}^N \minus \{0\}$, let the hyperplane $\Pi_p$ in the
affine space $\mathbb{A}^N$ be given by
\[x \in \Pi_p \quad \mbox{ iff } \quad \sum_{i=1}^N p_ix_i = 1.\]
Consider the family of hyperplanes $(\Pi_p)_{p \in \mathbb{A}^N\minus
\{0\}}$, which is the family of all affine hyperplanes which do not
pass through the origin. From the equation defining the hyperplanes it
follows that there is a duality: $a \in \Pi_p$ iff $p \in \Pi_a$.

The lemma we use is in the style of model-theoretic geometry, and is
adapted from part of a proof in \cite{Zilber04}. Here and later, $x \in \acl{X}$ means that $x$ is a point \emph{in some variety} which is algebraic over $X$. We do not restrict this notation to $x$ being an element or tuple from affine space.

\begin{lemma}\label{generic hyperplanes}
  Let $A$ be a field, let $g \in \mathbb{A}^N$ and let $p$ be generic
  in $\Pi_g$ over $A$. Suppose that $h$ is any tuple (a point in any
  algebraic variety) such that $h \in \acl(Ag)$. Then either $g \in
  \acl(Ah)$ or $\td(h/Ap) = \td(h/A)$ (that is, $h$ is independent
  of $p$ over $A$).
\end{lemma}
\begin{proof}
  If $g$ is algebraic over $A$ then the result is trivial, so we
  assume not.

  Let $U = \loc(p/\acl(Ah))$. Suppose $\td(h/Ap) < \td(h/A)$. Then, by
  counting transcendence bases, $\dim U = \td(p/Ah) < \td(p/A) = N$,
  the last equation holding because $g \notin \acl(A)$ and so $p$ is
  generic in $\mathbb{A}^N$ over $A$. But $\td(p/Ah) \ge \td(p/Ag) =
  N-1$ as $p$ is generic in $\Pi_g$, an $(N-1)$-dimensional variety
  defined over $Ag$. Hence $\dim U = N-1$. Now $\acl(Ah) \subs
  \acl(Ag)$, so $U = \loc(p/\acl(Ah) \sups \loc(p/\acl(Ag)) = \Pi_g$.
  But $\dim U = \dim \Pi_g$ and both $U$ and $\Pi_g$ are irreducible
  and Zariski-closed in $\mathbb{A}^N$, so $U = \Pi_g$.

  Hence $\Pi_g$ is defined over $\acl(Ah)$, and so is the set
  \[\class{x \in \mathbb{A}^N}{(\forall y \in \Pi_g)[x \in \Pi_y]} =
  \{g\}.\] Thus $g \in \acl(Ah)$.
\end{proof}

To have generic hyperplanes definable in the structure, we need to
know that it has large enough transcendence degree.
\begin{lemma}
  Suppose $\Ss \neq \{1\}$ and $M \models T_\Ss^U + \mathrm{NT} +
  \mathrm{EC}'$. Then $\td(M/C)$ is infinite.
\end{lemma}
\begin{proof}
  We build a tower of algebraically closed field extensions $C \subsetneq K_1 \subsetneq K_2 \subsetneq \cdots $ inside $M$. By NT, we can find a proper extension $K_1$ of $C$.

  Now suppose inductively that we have built the tower up to $K_i$ for some $i \ge 1$. Let $S \in \Ss$ be nontrivial, and let $x_i \in LS(K_i)\minus LS(K_{i-1})$. Let $V_i$ be the subvariety of $TS$ given by $V_i = \class{(x,y)\in LS \cross S}{x=x_i}$. Then $V_i$ is perfectly rotund, so, by EC$'$, there is $y_i \in S(M)$ such that $(x_i,y_i)\in \Gamma_S$. Let $K_{i+1} = K_i(y_i)^\alg$. By SP, $\td(K_{i+1}/C) \ge i\dim S + 1$ for each $i$. Thus $\td(M/C)$ is infinite.
\end{proof}

\begin{prop}\label{EC' implies EC}
  $\mathrm{EC}' \implies \mathrm{EC}$.
\end{prop}
\begin{proof}
  The proof is a sequence of reductions. Suppose $M \models T_\Ss^U$ +
  EC$'$. If $M = C$ or $\Ss = \{1\}$ then trivially $M \models$
  EC, so we assume $M \models$ NT and $\Ss \neq \{1\}$.

  Let $S \in \Ss$, and let $V \subs TS$ be rotund. We may assume that
  $V$ is irreducible.

  \Step{1: $\mathbf{\dim V = \dim S}$} 
  We first show that if $\dim V > \dim S$, we can find a subvariety
  $V'$ of $V$ which is still rotund and irreducible, with $\dim
  V' = \dim V - 1$. By induction, we can assume that $\dim V = \dim
  S$.

  Let $A$ be a subfield of $M$ which is a field of definition of $V$,
  with finite transcendence degree over $C$. Let $g \in V(M)$, generic
  over $A$. (Such a $g$ exists because $M$ is algebraically closed and
  has infinite transcendence degree over $C$, but we don't assume $g \in
  \Gamma_S$.)

  Although $TS$ will not in general be an affine variety, we can embed
  it in some affine space $\mathbb{A}^N$ as a constructible set in a
  way which preserves the notion of algebraic dependence. (This
  follows from the model-theoretic definition of a variety.) Now we
  consider $g$ as a point in $\mathbb{A}^N$, and choose $p$ in
  $\Pi_g(M)$ such that $p_1,\ldots,p_{N-1}$ are algebraically
  independent over $A(g)$. 

  Let $A' = A(p)^\alg$ and let $V' = \loc(g/A')$, the locus being
  meant as a subvariety of $V$, not of $\mathbb{A}^N$. Then $\dim V' =
  \dim V - 1$. We show that $V'$ is rotund.


  Let $\ra{S}{q}{H}$ be an algebraic quotient map, and consider the
  image $h = (Tq)(g)$ in $TH$. Then
  $h \in \acl(Ag)$, and $\dim (Tq)(V') = \td(h/A')$. If $g \in
  \acl(Ah)$ then
  \[\td(h/A') = \td(g/A') = \dim V - 1 \ge \dim S \ge \dim H.\]
  Otherwise, by lemma~\ref{generic hyperplanes}, $\td(h/A') =
  \td(h/A)$, so $\dim (Tq)(V') = \dim (Tq)(V)$ which is at least $\dim
  H$ by rotundity of $V$. Thus $V'$ is rotund.

  \Step{2: Perfect Rotundity} Now we have $V$ rotund, irreducible, and
  of dimension equal to $\dim S$. Again, let $A$ be a subfield of $M$
  which is a field of definition of $V$, of finite transcendence
  degree over $C$, and now assume $A$ is algebraically closed.

  Consider the extension $B$ of $A$ where $B = A(g)^\alg$, and $g \in
  \Gamma_S \cap V$, with $g$ generic in $V$ over $A$. The extension $A
  \into B$ is self-sufficient, since $V$ is rotund. Also $\delta(B/A)
  = 0$, as $\dim V = \dim S$. Split the extension up into a maximal
  chain of self-sufficient extensions 
  \[A = B_0 \strong B_1 \strong B_2 \strong \cdots \strong B_l = B\]
  with each $B_i$ algebraically closed and each inclusion proper. We
  show inductively that $B_i$ is realised in $M$ over $B_{i-1}$. 

  Let $b_i$ be a basis for $\Gamma(B_i/B_{i-1})$. We have
  $\delta(B_i/B_{i-1}) = 0$, because $B_{i-1} \strong B_i$ and $B_i
  \strong B$, so $\loc(b_i/B_{i-1})$ is free and rotund, and its dimension is
  equal to $\dim \Smax(B_i/B_{i-1})$. If
  $\ra{\Smax(B_i/B_{i-1})}{q}{H}$ were a proper nontrivial quotient and
  $\dim(Tq)(\loc(b_i/B_{i-1}) = \dim H$ then $B_{i-1}((Tq)(b_i))$ would
  be a self-sufficient extension intermediate between $B_{i-1}$ and
  $B_i$. By assumption, no intermediate extensions exist, and so
  $\loc(b_i/B_{i-1})$ is perfectly rotund.

  Let $(W_e)_{e \in Q(C)}$ be a parametric family of proper subvarieties of $V$, the family defined over $C_0$, and $S \rOnto^f \Smax(B_1/A)$. Then $Tf(W_e)_{e \in Q(C)}$ is a parametric family of proper subvarieties of $Tf(V)$. Since $f$ is defined over $C_0$, so is this family. Hence, by EC$'$, there is $b_1' \in Tf\left(V \minus \bigcup_{e \in Q(C)}W_e\right)$. Replacing $A$ by $B_1$, we inductively construct $b' \in V \minus \bigcup_{e \in Q(C)}W_e$. Thus $M \models$ EC.
\end{proof}

\begin{prop}\label{SEC lemma}
  The Fraiss\'e limit $U$ satisfies SEC and ID.
\end{prop}
\begin{proof}
  Let $V$ be a rotund subvariety of $TS$, defined over a finitely
  generated subfield $A$ of $U$. Let $g$ be a generic point of $V$
  over $A$. Let $B$ be the extension of $A$ defined by taking $g$ as a
  basis for $\Gamma(B/A)$. Since $V$ is rotund, the extension $A \into
  B$ is self-sufficient.

  The hull $\hull{A}$ of $A$ has finite transcendence degree over $A$,
  so by theorem~\ref{free amalgamation} there is a free amalgam $E$ of
  $\hull{A}$ and $B$ over $A$ such that $\hull{A} \strong E$. Hence,
  by the $\Kso$-saturation of $U$, there is an embedding $\theta$ of
  $E$ into $U$ over $\hull{A}$. Then $\theta(g) \in \Gamma_S \cap V$,
  so $U$ satisfies SEC.

  For $n \in \N$, let $A_n$ be an algebraically closed field extension
  of transcendence degree $n$ over $C$, and for each $S \in \Ss$, let
  $\Gamma_S = TS(C)$. So there are no points of $\Gamma$ outside
  $C$. Then each $A_n \in \K$, so it embeds self-sufficiently into
  $U$. The dimension of $A_n$ is $n$, and self-sufficient embeddings
  preserve the dimension, so the dimension of $U$ is at least $n$ for
  every $n \in \N$. Hence it is infinite.
\end{proof}

\begin{theorem}\label{uniqueness of U}
  The Fraiss\'e limit $U$ is the unique countable model of $T_\Ss^U$ which
  satisfies SEC and ID and has $\td(C/C_0) = \aleph_0$.
\end{theorem}

\begin{proof}
  The case where $\Ss = \{1\}$ is trivial, so we assume $\Ss \neq
  \{1\}$. Let $M$ be any such model. We will show that $M$ is
  $\Kso$-saturated. The result follows by theorem~\ref{amalgam is
    unique}. Let $A$ be a self-sufficient finitely generated
  substructure of $M$ and let $A \strong B$ be a self-sufficient
  extension with $B$ finitely generated. We must show that $B$ can be
  embedded self-sufficiently in $M$ over $A$.

  By lemma~\ref{finitely generated extensions}, the extension $B$ of
  $A$ is determined by the group $S = \Smax(B/A)$, the locus
  $\loc_A(g) \subs \Smax(B/A)$ of a basis $g$ for $\Gamma(B/A)$, and
  the natural number $t = \td(B/A(g))$. Suppose that $b$ is a
  transcendence base for $B/A(g)$. Take $S' \in \Ss$ of dimension at
  least $t$ and extend $b$ to an algebraically independent tuple $b'
  \in \ga^{\dim S'}$. Take $s \in S'$ generic over $B(b')$. Then there
  is a self-sufficient extension $B \rStrong B'$ generated by $(b',s)$
  such that $(b',s) \in \Gamma_{S'}$. By replacing $B$ by $B'$, and $S$
  by $S \cross S'$, we may assume that $t=0$, that is, that $B$ is
  generated by $g$ over $A$.

  Let $V = \loc_A(g)$. Then $V$ is rotund and irreducible. We use the
  method of step 1 of the proof of \ref{EC' implies EC} above to
  reduce to replace $V$ by a subvariety $V'$ with $\dim V' = \dim
  \Smax(B/A)$, with $V'$ also rotund and irreducible. However, for
  each generic hyperplane $\Pi_p$, by ID we may choose the
  $p_1,\ldots,p_{N-1}$ not just to be algebraically independent, but
  in fact $\cl$-independent. Let $A'$ be the extension of $A$
  generated by all the $p_i$, for each hyperplane used. Then $A'$ is
  generated over $A$ by $\cl$-independent elements, and hence $A
  \strong A'$ and $A' \strong M$.

  By SEC, there is $h \in \Gamma_S \cap V'$ in $M$, generic in $V'$
  over $A'$. Thus $h$ is also generic in $V$ over $A$. Let $B'' =
  \gen{A'h}$. Then $\delta(B''/A') = 0$, and so $B'' \strong M$. Also
  $B' \leteq \gen{Ah}$ is isomorphic to $B$ over $A$, and $B' \strong
  M$ as $\td(B''/B') = d(B''/B')$. Hence $M$ is $\Kso$-saturated, and
  $M \iso U$, as required.
\end{proof}

\medskip

\begin{defn}\label{defn of T_S}
  Let $T_\Ss$ be the theory $T_\Ss^U$ + EC + NT, that is, U1 --- U7 +
  SP + EC + NT.
\end{defn}
We have already seen (\ref{U axioms are first order}) that U1 --- U7
are expressible as first order axiom schemes, and NT is a first order
axiom.  In section~\ref{fot section} we will show that SP and EC are
also expressible as first order schemes, so $T_\Ss$ is axiomatizable
as a first order theory. We will also show that $T_\Ss$ is complete.


\section{Reducts of differential fields}\label{diff fields section}

\subsection{Differential forms in differential algebra}

Given a field $C$ and a $C$-algebra $A$, we form the $A$-module
$\Omega(A/C)$ of K\"ahler differentials as in \cite{Shafarevich1} or
\cite[p386]{Eisenbud}. If $A$ is a field, $F$, we can identify the
$F$-vector space $\Der(F/C)$ of derivations on $F$ which are constant
on $C$ with the dual space of $\Omega(F/C)$, by means of the universal
property of $\Omega$. If $\omega \in \Omega(F/C)$ and $D \in
\Der(F/C)$ we write $D^*$ for the associated element of
$\Omega(F/C)^*$.

Let $V$ be an irreducible affine variety defined over a field $C$, and
let $A$ be the coordinate ring of $V$, a $C$-algebra. If $F$ is a
field extension of $C$, an $F$-point $x$ of $V$ is associated with a
$C$-algebra homomorphism $\ra{A}{x}{F}$, and by functoriality of
$\Omega$ this defines a map
\map{\Omega(A/C)}{x_*}{\Omega(F/C)}{\omega}{\omega(x)}

More generally, if $V$ is not affine (for example, $V$ is an abelian
variety) we replace $A$ by the sheaf of coordinate rings on $V$, and
consider the module of global differentials which we write
$\Omega[V]$. Again, an $F$-point of $V$ defines a map
\map{\Omega[V]}{x_*}{\Omega(F/C)}{\omega}{\omega(x)}

Allowing $x$ to vary over $V(F)$ gives a map \map{V(F) \cross
\Omega[V]}{}{\Omega(F/C)}{(x,\omega)}{\omega(x)} and fixing $\omega$
gives a map which we write $\ra{V(F)}{\omega}{\Omega(F/C)}$.

If $V$ is a commutative algebraic group $G$ then $\Omega[G]$ is spanned by
a basis of invariant differential forms. These forms are related to
the logarithmic derivative.
\begin{lemma}\label{diff char is homo}
  If $\zeta \in \Omega[G]$ is an invariant differential form then the
  map 
  \map{G(F)}{\zeta}{\Omega(F/C)}{x}{\zeta(x)} is a group homomorphism.

  If $\zeta_1,\ldots,\zeta_n$ is a basis of invariant forms of
  $\Omega[G]$, then the logarithmic derivative $\logd_G(x) =
  \tuple{D^*\zeta_1(x),\ldots,D^*\zeta_n(x)}$.
\end{lemma}
\begin{proof}
  This is a restatement of the last result from \S3 of
  \cite{MarkerMK}. The first part is due to Rosenlicht
  \cite{Rosenlicht57}.
\end{proof}

In \cite{Ax71}, Ax used the \emph{Lie derivative} without naming or
defining it explicitly, and we will use it for the same purpose. Many
differential geometry books give an account of the Lie derivative in
that context, but for clarity we include a description for this
algebraic context.

Rewriting $\Omega(F/C)$ as $\Omega^1(F/C)$, the map
$\ra{F}{d}{\Omega^1(F/C)}$ can be thought of as the coboundary map in
the de Rham complex
\begin{diagram}
  0 & \rTo & F = \Omega^0(F/C) & \rTo^d & \Omega^1(F/C) & \rTo^d
  &\Omega^2(F/C) & \rTo^d & \cdots &.
\end{diagram}
We write $\Omega^\bullet(F/C)$ for the union of the complex.

For any derivation $D \in \Der(F/C)$, the map
$\ra{\Omega^1(F/C)}{D^*}{F}$ defined previously extends to a map
$\ra{\Omega^\bullet(F/C)}{D^*}{\Omega^\bullet(F/C)}$ which is defined
for $\omega \in \Omega^n(F/C)$ by
\[(D^*\omega)(D_1,\ldots,D_{n-1}) = \omega(D,D_1,\ldots,D_{n-1}).\]

This map $D^*$ has degree $-1$, that is if $\omega \in \Omega^n(F/C)$
then $D^*\omega \in \Omega^{n-1}(F/C)$. By definition, $d$ has degree
$+1$. These operations can be combined into an operation of degree 0
\[L_D = D^* \circ d + d \circ D^*\]
called the \emph{Lie derivative} of $D$ on $\Omega^\bullet(F/C)$.

\begin{lemma}\label{Lie derivative properties}
  The Lie derivative $L_D$ has the following properties. Let $\omega
  \in \Omega^1(F/C)$, $D,D' \in \Der(F/C)$, and $a \in F$.
  \begin{enumerate}
  \item $L_D$ is $C$-linear.
  \item $(L_D\omega) D' = D(\omega D') - \omega[D,D']$
  \item $L_D(a\omega) = (Da)\omega + a(L_D \omega)$
  \end{enumerate}
\end{lemma}
\begin{proof}
  1. is immediate, since $d$ and $D^*$ are $C$-linear. For 2,
  \begin{eqnarray*}
    (L_D\omega) D' &=&(D^* d\omega)D' + (d(\omega D))D'\\
    &=& (d\omega)(D,D') + D'(\omega D)\\
    &=& D(\omega D') - D'(\omega D) - \omega[D,D'] +
    D'(\omega D)\\
    &=&  D(\omega D') - \omega[D,D']
  \end{eqnarray*} 
  and for 3,
  \begin{eqnarray*}
    L_D(a\omega) D' &=&  D(a\omega D') - a\omega[D,D']\\
    &=& (Da)\omega D' + aD(\omega D') - a \omega[D,D']\\
    &=& (Da)\omega D' + a(L_D \omega)D'.
  \end{eqnarray*} 
\end{proof}

A standard fact which we need is that invariant differential forms are
closed in the sense of de Rham cohomology.
\begin{lemma}\label{invariant form is closed}
  Let $G$ be a commutative algebraic group defined over $C$. Let
  $\omega \in \Omega[G]$ be an invariant differential form on $G$, and
  let $x\in G(F)$. Then $\omega(x)$ is a closed K\"ahler differential
  in $\Omega(F/C)$, that is, $d\omega(x) = 0$ in $\Omega^2(F/C)$.
\end{lemma}
We give a proof for completeness. See for example \cite{MarkerMK} for
notation.
\begin{proof}
  The Lie algebra $L$ of $G(C)$ is canonically isomorphic to the space
  of invariant vector fields on $G(C)$, and is a $C$-vector space of
  dimension $n = \dim G$. Let $X_1,\ldots,X_n$ be a basis of $L$. The
  vector space $\Der(F/C)$ is canonically isomorphic to the space of
  all $F$-valued invariant vector fields on $G(C)$, which is $L
  \otimes_C F$, so $X_1,\ldots,X_n$ also forms an $F$-basis of
  $\Der(F/C)$. Let $D_1,D_2 \in \Der(F/C)$, say $D_1 = \sum_{i=1}^n
  a_iX_i$ and $D_2 = \sum_{i=1}^n b_iX_i$ with the $a_i,b_i \in F$.
  Then
  \begin{eqnarray*}
    d\omega(D_1,D_2) & = & d\omega\left(\sum_{i=1}^n a_iX_i, \sum_{i=1}^n
    b_iX_i\right)\\
    & = & \sum_{i,j}a_i b_j d\omega(X_i,X_j) \quad \mbox{by bilinearity of
    } d\omega\\
    & = & \sum_{i,j}a_i b_j (X_i(\omega X_j) - X_j(\omega X_i) -
    \omega[X_i,X_j]) \quad .
  \end{eqnarray*}
  Now $\omega$ and $X_j$ are both invariant, so for any $x,y \in G(C)$,
  \begin{eqnarray*}
    (\omega X_j)_{xy} & = & \omega_{xy}(X_j)_{xy}\\
    & = & \omega_y(d\lambda^{x^{-1}}_xy(X_j)_{xy})\\
    & = & \omega_y(d\lambda^{x^{-1}}_xy d\lambda^x_y (X_j)_y)\\
    & = & \omega_y(X_j)_y\\
    & = & (\omega X_j)_y
  \end{eqnarray*}
  and so $\omega X_j$ is a constant scalar field on $G(C)$. Thus
  $X_i(\omega X_j) = 0$, and similarly $X_j(\omega X_i) = 0$. So
  \[d\omega(D_1,D_2) = - \sum_{i,j}a_i b_j \omega[X_i,X_j]\]
  but $[\ ,\ ]$ is the bracket on the Lie algebra of $G$, and $G$ is
  commutative so the bracket is identically zero. So $d\omega(D_1,D_2)
  = 0$ for all $D_1,D_2 \in \Der(F/C)$, and hence $d\omega = 0$.
\end{proof}

\subsection{The algebraic axioms}

The vector space $\Omega[G]$ is associated with the cotangent space of
$G$ at the identity, that is, with the dual of $LG$. Thus the
canonical isomorphism between $LG$ and $LLG$ gives rise to a canonical
isomorphism between $\Omega[G]$ and $\Omega[LG]$. 

Let $(\zeta_1,\ldots,\zeta_n)$ be a basis of the space of invariant
differential forms on $S$ and let $(\xi_1,\ldots,\xi_n)$ be the
corresponding basis of the space of invariant differential forms on
$LS$. Write $\omega_i(x,y) = \zeta_i(y) - \xi_i(x)$, for each
$i=1,\ldots,n$.

Recall that the tangent bundle $TS$ is identified with $LS \cross S$, and $\Gamma_S$ is defined by \[(x,y) \in \Gamma_S \iff \logd_{LG}(x) = \logd_{G}(y).\]
Translating the definition of the logarithmic derivatives into coordinates using
lemma~\ref{diff char is homo} gives us an alternative characterization of $\Gamma_S$ in terms of the differential forms $\omega_i$.
\begin{lemma}\label{Gamma in coords}
  Let $x \in LS(F)$ and $y\in S(F)$, and let the differentials
  $\omega_i$ be defined as above. Then $(x,y) \in \Gamma_S$ iff for
  each $i=1,\ldots,n$, the equation $D^*\omega_i(x,y) = 0$ holds. \qed
\end{lemma}

Consider a differential field of characteristic zero
$\tuple{F;+,\cdot,D}$, and let $C$ be the constant field. As described
in the introduction, we consider the reduct of $F$ to the language
$\tuple{F;+,\cdot,C,(\Gamma_S)_{S \in\Ss}, (\hat{c})_{c\in C_0}}$. We
also consider a slight generalization. Suppose now that $F$ is a field
with a family $\Delta$ of derivations, such that $C = \bigcap_{D \in
\Delta} \ker D$. For each $D \in \Delta$, we can consider the
solution set $\Gamma_{S,D}$ of the exponential differential equation
for $S$ with respect to $D$. Write $\Gamma_S = \bigcap_{D\in
\Delta}\Gamma_{S,D}$.

Given a finite set of derivations $\Delta = \{D_1,\ldots,D_r\}$ on
$F$, and a tuple $a = \tuple{a_1,\ldots,a_n}$ from $F$, define the
Jacobian matrix of $a$ with respect to $\Delta$ to be
\[\Jac_\Delta(a) = \begin{pmatrix}
  D_1 a_1 & \cdots & D_1 a_n\\ 
  \vdots & \ddots & \vdots\\ 
  D_r a_1 & \cdots & D_r a_n
\end{pmatrix}\]
and write $\rk \Jac_\Delta(a)$ to be the rank of this matrix. 

If $\Delta$ is an infinite set of derivations, the rank of the
Jacobian matrix is then defined to be 
\[\rk\Jac_\Delta(a) = \max\class{\rk\Jac_{\Delta'}(a)}{\Delta'
 \mbox{ is a finite subset of } \Delta}.\] 
The rank of the matrix is bounded by the number $n$ of columns, so
this maximum is well defined. We will not usually write the dependence
on $\Delta$ explicitly, so will write this simply as $\rk\Jac(a)$.

\begin{prop}\label{reduct satisfies U axioms}
  Let $\tuple{F;+,\cdot,D}$ be a differential field, let $C_0$ be a
  subfield of the field of constants $C$, and let $\Ss$ be a
  collection of semiabelian varieties, each defined over $C_0$.
  Then the reduct $\tuple{F;+,\cdot,C,\{\hat{c}\}_{c \in
  C_0},\{\Gamma_S\}_{S\in \Ss}}$ satisfies the axioms U2---U7 and
  U1$'$, the universal part of U1.
\end{prop}
\begin{proof}
  Axiom U1$'$ says that $F$ is a field, $C$ is a relatively
  algebraically closed subfield, and the constants $\hat{c}$ have the
  correct algebraic type, all of which holds in the reduct.

  For U2, $\Gamma_S$ is the kernel of the group homomorphism 
  \map{TS(F)}{}{LS(F)}{(x,y)}{\logd_S(y) - \logd_{LS}(x)}
  and so is a subgroup of $TS$.

  The logarithmic derivatives $\logd_S$ and $\logd_{LS}$ vanish on the
  $C$-points of $S$ and $LS$ respectively, so $TS(C) \subs \Gamma_S$,
  which is U3.

  The fibre of $x=0$ is $\class{y\in S(F)}{\logd_S(y)=0}$ which is
  $S(C)$. Similarly, the fibre of $y=0$ is $LS(C)$. This is axiom U4.

  Suppose that $\ra{S_1}{f}{S_2}$ is an algebraic group homomorphism,
  and let $(x,y) \in \Gamma_{S_1}$. Let $\zeta$ be an invariant
  differential form on $S_2$, let $\xi$ be the corresponding invariant
  form on $LS_2$, and let $\omega = \zeta - \xi$. To show $Tf(x,y) \in
  \Gamma_{S_2}$, it suffices to show that $D^*\omega(Tf(x,y)) = 0$. But
  \begin{equation*}
    D^*\omega(Tf(x,y)) = D^*\zeta(f(x)) - D^*\xi(df_e(y))\\ 
    = D^*(f_*\zeta)(y) - D^*({df_e}_*(\xi))(x)
  \end{equation*}
  where $f_*$ and ${df_e}_*$ denote the images of $f$ and $df_e$ under
  the contravariant cotangent bundle functor. The image of an
  invariant form is an invariant form, and so $f_*\zeta$ and
  ${df_e}_*(\xi)$ are corresponding invariant differential forms on
  $S_1$ and $LS_1$. Hence $D^*\omega(Tf(x,y)) = 0$, since $(x,y) \in
  \Gamma_{S_1}$.

  Now suppose that $f$ is an isogeny. Let $(v,w) \in \Gamma_{S_2}$ and
  let $(x,y)\in TS_1$ such that $Tf(x,y) = (v,w)$. Let $\zeta$ be an
  invariant form on $S_1$, let $\xi$ be the corresponding
  invariant form on $LS_1$, and let $\omega = \zeta - \xi$ on
  $TS_1$. Since $f$ is an isogeny, the map $Tf_*$ is an isomorphism
  between the spaces of invariant forms on $TS_2$ and $TS_1$. Let
  $\eta = (Tf_*)^{-1}(\omega)$. Now
  \begin{equation*}
    D^*(\omega(x,y)) = D^*(Tf_* \eta)(x,y) = D^*\eta((Tf)(x,y))\\ 
    = D^*\eta(v,w) = 0
  \end{equation*}
  so $(x,y) \in \Gamma_{S_1}$. This proves axiom U5.

  If $S_1 \subs S_2$, let $\zeta_1,\ldots,\zeta_m$ be a basis of
  invariant differential forms on $S_1$ and extend to a basis
  $\zeta_1,\ldots,\zeta_n$ of invariant differential forms on
  $S_2$. Let $\xi_1,\ldots,\xi_n$ be the corresponding basis of
  invariant differential forms on $LS_2$, and let $\omega_i(x,y) =
  \zeta_i(y)-\xi_i(x)$. If $g\in \Gamma_{S_1}$ then $D^*\omega_i(g) =
  0$ for $i=1,\ldots,m$ by definition of $\Gamma_{S_1}$ and for
  $i=m+1,\ldots,n$ because $g \in TS_1$. So $g \in
  \Gamma_{S_2}$. Conversely, if $g \in \Gamma_{S_2} \cap TS_1$ then
  $D^*\omega_i(g) = 0$ for $i=1,\ldots,m$ and so $g\in
  \Gamma_{S_1}$. So U6 holds.

  The logarithmic derivative of a product is given componentwise, that
  is, $\logd_{G_1 \cross G_2}(g_1,g_2) =
  (\logd_{G_1}(g_1),\logd_{G_2}(g_2))$. Axiom U7 follows.
\end{proof}

\subsection{The Schanuel property}

Next we prove the Schanuel property, in a slightly stronger form for
differential fields with a family of derivations. The following lemma
on algebraic subgroups of $TS$ is central to the proof.
\begin{lemma}\label{subgroups of TS}
  Let $S$ be a semiabelian variety, and let $G$ be an algebraic
  subgroup of $TS = LS \cross S$. Then $G$ is of the form $G_1 \cross
  G_2$ for some subgroup $G_1$ of $LS$ and some subgroup $G_2$ of $S$.
\end{lemma}
\begin{proof}
  Let $G_1 = \pr_{LS}(G)$ and $G_2 = \pr_S(G)$. Write $0$ for the
  identity element of $LS$ and $1$ for the identity element of
  $S$. Define subgroups $H_1 = \class{x \in G_1}{(x,1)\in G}$ and $H_2
  = \class{y \in G_2}{(0,y)\in G}$ and define a quotient map $\ra{G_2}{\theta}{G_1/H_1}$
  by $\theta(y) = \class{x \in G_1}{(x,y)\in G}$. It is easy to check
  that $\theta$ is a regular group homomorphism with kernel $H_2$.

  $G_1/H_1$ is a vector group, since algebraic subgroups and quotients
  of vector groups are vector groups. $G_2$ is an algebraic subgroup
  of a semiabelian variety, so is semiabelian-by-finite. But the only
  regular homomorphism from a semiabelian-by-finite group to a vector
  group is the zero homomorphism, so $H_2 = G_2$, and thus also $H_1 =
  G_1$. Hence $G = G_1 \cross G_2$.
\end{proof}

We separate out the following intermediate step from the proof of the
Schanuel property, as it will also be used later to prove EC. Recall
the definition of the $\omega_i$ from before lemma~\ref{Gamma in
  coords}.
\begin{prop}\label{Main Schanuel prop}
  Suppose $(x,y) \in \Gamma_S$ and the differentials
  $\omega_1(x,y),\ldots,\omega_n(x,y)$ are $F$-linearly dependent in
  $\Omega(F/C)$. Then there is a proper algebraic subgroup $H$ of $S$
  and a point $\gamma \in TS(C)$ such that $(x,y)$ lies in the coset
  $\gamma \cdot TH$.
\end{prop}

\begin{proof}

\begin{step}[Step 1: $C$-linear dependence]
  Take $\alpha_i \in F$ such that $\sum_{i=1}^n \alpha_i \omega_i(x,y)
  = 0$ is a minimal $F$-linear dependence on the $\omega_i$, that is,
  if $I = \class{i}{\alpha_i \neq 0}$ then the $F$-linear dimension of
  $\class{\omega_i}{i \in I}$ is $|I| - 1$. Dividing by some non-zero
  $\alpha_i$, we may assume that for some $i = i_0$, $\alpha_{i_0} =
  1$.

  Applying the Lie derivative $L_{D}$ for $D \in \Delta$ we get
  \begin{eqnarray*}
    0 = L_{D} \sum_{i=1}^n \alpha_i \omega_i(x,y) & = & \sum_{i=1}^n
    [(D\alpha_i)\omega_i(x,y) + \alpha_i L_{D} \omega_i(x,y)]\\
    & = & \sum_{i=1}^n [(D\alpha_i) \omega_i(x,y) + \alpha_i(
    d D^* \omega_i(x,y) + D^* d \omega_i(x,y))]\\
    & = & \sum_{i=1}^n (D\alpha_i) \omega_i(x,y)
  \end{eqnarray*}
  using the properties of the Lie derivative given in lemma~\ref{Lie
  derivative properties}. The last equality uses the fact that $(x,y)
  \in \Gamma_S$, and so $D^* \omega_i(x,y) = 0$ for each $i$. It also
  uses the fact that each differential $\omega_i(x,y)$ is a difference
  of invariant differentials, and hence is closed by
  lemma~\ref{invariant form is closed}, so $d\omega_i(x,y) =
  0$ for each $i$.

  Now $\alpha_{i_0} = 1$, so $D\alpha_{i_0} = 0$ but then, by the
  minimality of set $I$, we have that $D\alpha_i = 0$ for every $i$
  and each $D \in \Delta$, so each $\alpha_i \in C$. Hence the
  $\omega_i(x,y)$ are $C$-linearly dependent.
\end{step}

\begin{step}[Step 2:\footnote{Thanks to Piotr Kowalski for an improved
    argument in step 2.} A subgroup of $TS$]
  Let $\eta = \sum_{i=1}^n \alpha_i \omega_i$. Then $\eta$ is an
  invariant differential form on $TS$, defined over $C$.

  By lemma~\ref{diff char is homo}, $\eta$ defines a group
  homomorphism $\ra{TS}{}{\Omega(F/C)}$, so $\ker \eta$ is a subgroup
  of $TS$. The $\omega_i$ are linearly independent, so $\eta \neq 0$
  and hence $\ker \eta$ is a proper subgroup of $TS$. By construction,
  $(x,y) \in \ker \eta$.

  Let $V = \loc_C(x,y)$, the algebraic locus of $(x,y)$ over $C$, and
  an algebraic subvariety of $TS$. The field $C$ is algebraically
  closed, so $V$ has a $C$-point, say $\gamma = (\gamma_1,\gamma_2)$,
  with $\gamma_1 \in LS$ and $\gamma_2 \in S$. Let $V' = \class{v
  \gamma^{-1}}{v \in V}$. Then $V'$ is an irreducible algebraic
  variety defined over $C$, containing the identity of $TS$, and
  having $(x',y') = (x\gamma_1^{-1},y\gamma_2^{-1})$ as a generic
  point over $C$.

  Let $O$ be the orbit of $(x',y')$ in the algebraic closure $\bar{F}$
  of $F$, under $\Aut(\bar{F}/C)$, that is, automorphisms of the pure
  field.

  For $n \in \N$, let $nV' = \class{v_1\cdot\,\cdots\,\cdot v_n}{v_i
    \in V'}$, and similarly $nO$. By the indecomposability theorem due
  to Chevalley \cite[Chapter II, section 7]{Chevalley51} (see also
  \cite[p261]{Marker02}), there is $n \in \N$ such that $nV' = G$, an
  algebraic subgroup of $TS$. Now $nO \subs nV'$ (where now we
  identify $nV'$ with its $\bar{F}$-points) and $O$ contains all
  realizations of the generic type of $V'$, that is, of $\tp(x/C)$,
  hence $nO$ contains all the realizations of the generic types of
  $nV'$. Every element of $G$ is the product of two generic elements,
  so $2nO = G$.

  The differential form $\eta$ vanishes on $TS(C)$, so
  \[\eta(x',y') = \eta(x,y) - \eta(\gamma_1,\gamma_2) = 0\] 
  but then $\eta$ vanishes on $O$, because $\eta$ is defined over $C$
  and hence its kernel is $\Aut(\bar{F}/C)$-invariant. Since $\eta$ is
  a group homomorphism, it vanishes on the subgroup $G$ generated by
  $O$, that is $G \subs \ker \eta$. Since $\ker \eta$ is a proper
  subgroup of $TS$, $G$ is a proper algebraic subgroup of $TS$. By
  lemma~\ref{subgroups of TS}, $G$ is of the form $J \cross H$, with
  $J$ a subgroup of $LS$ and $H$ a subgroup of $S$.
\end{step}

\begin{step}[Step 3: A subgroup of $S$]
  Recall that $\omega_i(x,y) = \zeta_i(y) - \xi_i(x)$, with the
  $\zeta_i$ being invariant forms on $S$ and the $\xi_i$ being
  invariant forms on $LS$. Let $\nu = \sum_{i=1}^n \alpha_i \zeta_i$
  and $\mu = \sum_{i=1}^n \alpha_i \xi_i$. 

  For any $h \in H$, 
  \[\nu(h) = \nu(h) - \mu(0) = \eta(0,h) = 0\]
  because $(0,h) \in G \subs \ker \eta$. Thus $H \subs \ker \nu$. Now
  $\nu$ is a nonzero invariant form on $S$, since the $\zeta_i$ are
  linearly independent. Hence $H$ is a proper algebraic subgroup of
  $S$.
\end{step}

\begin{step}[Step 4: Constant cosets]
  Consider the quotient group $\Gamma_S/TS(C)$. By axiom U4,
  it is the graph of a bijection 
  \[\ra{\frac{\pr_{LS} \Gamma_S}{LS(C)}}{\theta}{\frac{\pr_S
      \Gamma_S}{S(C)}}\]
  where $\pr_{LS}$ is the projection $\ra{TS}{}{LS}$ and $\pr_S$ is the
  projection $\ra{TS}{}{S}$. By the choice of corresponding bases of
  invariant forms $\zeta_1,\ldots,\zeta_n$ on $S$ and
  $\xi_1,\ldots,\xi_n$ on $LS$, and lemma~\ref{Gamma in coords}, 
  \[\theta^{-1}((\pr_1 \Gamma_S \cap H) \cdot S(C)) = (\pr_2 \Gamma_S
  \cap LH) \cdot LS(C)\]
  By construction of $H$, $y$ lies in a constant coset of $H$, and
  $(x,y) \in \Gamma_S$, so $\theta^{-1}(y \cdot S(C)) = x \cdot LS(C)$,
  hence $x$ lies in a constant coset of $LH$. Thus $(x,y)$ lies in a
  constant coset of $TH$, as required.
\end{step}
\end{proof}

\begin{theorem}[The Schanuel property]\label{Schanuel property}
  Let $F$ be a field of characteristic zero, let $\Delta$ be a
  collection of derivations on $F$, and let $C$ be the intersection of
  their constant fields. Let $S$ be a semiabelian variety defined over
  $C$, of dimension $n$, and let $\Gamma_S \subs LS\cross S$ be the
  solution set of the exponential differential equation of $S$ (that
  is, the intersection of the solution sets for each $D \in \Delta$).

  Suppose that $(x,y) \in \Gamma_S$ and $\td(x,y/C) - \rk\Jac(x,y) <
  n$. Then there is a proper algebraic subgroup $H$ of $S$ and a
  constant point $\gamma \in TS(C)$ such that $(x,y)$ lies in the
  coset $\gamma \cdot TH$.
\end{theorem}

\begin{proof}
  To prove the theorem, it suffices by proposition~\ref{Main Schanuel prop} to show that
  the differential forms $\omega_1(x,y),\ldots,\omega_n(x,y)$ are
  $F$-linearly dependent in $\Omega(F/C)$.

  Let $E = C(x,y)$, the subfield (not differential subfield) of $F$
  generated over $C$ by $x$ and $y$. Choose a finite tuple
  $D_1,\ldots,D_r$ of derivations from $\Delta$ such that the rank of
  the Jacobian matrix $\rk \Jac_\Delta (x,y)$ is equal to $\rk
  \Jac_{D_1,\ldots,D_r}(x,y)$. Write $D$ for the tuple
  $(D_1,\ldots,D_r)$, a map $\ra{F}{D}{F^r}$. Consider the diagram
  below, where $D^*$ is the $F$-linear map which comes from the
  universal property of $d$.
  \begin{diagram}
    F && \rTo^{d} && \Omega(F/C)\\
    & \rdTo(4,2)_{D} &&& \dTo>{D^*}\\
    && && F^r
  \end{diagram}
  Write $\Ann(D)$ for the kernel of the linear map $D^*$. The diagram
  restricts to
  \begin{diagram}
    E && \rTo^{d} && \Omega(E/C) \tensor_E F\\
      & \rdTo(4,2)_{D} &&& \dTo>{D^*}\\
      & &&& F^r
  \end{diagram}
  where again $D^*$ is $F$-linear, with kernel $(\Omega(E/C)\tensor_E
  F) \cap \Ann(D)$.

  The $E$-vector space $\Omega(E/C)$ has $E$-linear dimension
  equal to $\td(x,y/C)$, and so $\Omega(E/C) \otimes_E F$ has $F$-linear
  dimension also equal to $\td(x,y/C)$. 

  The image of $D^*$ is the image of $D$, which is spanned by the
  columns of the matrix $\Jac(x,y)$. Thus $\rk \Jac(x,y)$ is equal to
  the rank of the linear map $D^*$, which by the rank-nullity theorem
  is equal to the codimension of its kernel. Thus $\Omega(E/C)
  \tensor_E F \cap \Ann(D)$ has dimension $\td(x,y/C) - \rk\Jac(x,y)$,
  which by assumption is strictly less than $n$.

  The differential forms $\omega_i$ are defined over $C$, so each of
  the $n$ differentials $\omega_i(x,y)$ lies in $\Omega(E/C)$. Since
  $(x,y) \in \Gamma_S$, each $\omega_i(x,y)$ also lies in
  $\Ann(D)$. Hence they are $E$-linearly dependent, and in particular
  they are $F$-linearly dependent.
\end{proof}

\begin{cor}\label{reduct satisfies SP}
  The reduct of a differential field to the language $\L_\Ss$
  satisfies the SP axiom.
\end{cor}
\begin{proof}
  The axiom SP is just the special case of theorem~\ref{Schanuel
  property} for the semiabelian varieties which lie in $\Ss$, with
  $\Delta$ being the singleton $\{D\}$.
\end{proof}

\subsection{Existential closedness}

\begin{theorem}\label{reduct satisfies EC}
  Let $F$ be a differentially closed field (of characteristic zero,
  with one derivation). Then the reduct of $F$ to the language
  $\L_\Ss$ has the EC property.
\end{theorem}

\begin{proof}
  Let $S \in \Ss$, let $n = \dim S$, and let $V$ be a perfectly rotund
  subvariety of $TS$, defined over $F$. Let $(W_e)_{e \in Q(C)}$ be a parametric family of proper subvarieties of $V$, defined over $C_0$. We show there is $(x,y) \in \Gamma_S \cap V \minus \bigcup_{e \in Q(C)} W_e$. By proposition~\ref{EC' implies EC}, this suffices to prove the EC property.

  Let $D_0$ be the derivation on $F$. Let $(x,y)$ be a generic point
  of $V$ over $F$, and let $K = F(x,y)^{\alg}$, the algebraic closure of $F(x,y)$.

  We wish to consider the derivations in $\Der(K/C)$ which extend $D_0$
  on $F$. These form a coset of the subspace $\Der(K/F)$ of
  $\Der(K/C)$. In order to work with subspaces rather than cosets, we follow
  \cite{Pierce03} in defining
  \[\DKD = \class{D \in \Der(K/C)}{\exists \lambda \in K,
    D\restrict{F} = \lambda D_0}\]
  which can be considered as the dual space of a quotient $\OKD$ of
  $\Omega(K/C)$. This gives a sequence of inclusions
  \begin{diagram}
    \Der(K/F) & \rInto && \DKD & \rInto && \Der(K/C)
  \end{diagram}
  and dually surjections
  \begin{diagram} 
    \Omega(K/C) & \rOnto && \OKD & \rOnto && \Omega(K/F)
  \end{diagram}
  of $K$-vector spaces.
  
  We can consider the differentials $\omega_i(x,y)$ in $\Omega(K/C)$,
  and also in $\OKD$ and $\Omega(K/F)$ via the canonical surjections
  above. By the rotundity of $V$ and the genericity of $(x,y)$ in $V$
  over $F$, $(x,y)$ does not lie in an $F$-coset of $TH$ for any
  proper algebraic subgroup $H$ of $S$. Hence, by the contrapositive
  of proposition~\ref{Main Schanuel prop}, the differentials
  $\omega_1(x,y), \ldots, \omega_n(x,y)$ are $K$-linearly independent
  in $\Omega(K/F)$, and hence also in $\OKD$ and $\Omega(K/C)$.
  
  The $K$-linear dimension of $\OKD$ is equal to that of $\DKD$, which
  is $\dim \Der(K/F) + 1$, the ``+1'' because $F \neq C$. As $V$ is
  perfectly rotund it has dimension $n$ and, because $(x,y)$ is a
  generic point of $V$ over $F$ and $K = F(x,y)$, we have $\dim
  \Der(K/F) = n$.

  Let $\Lambda = \tuple{\omega_1(x,y),\ldots,\omega_n(x,y)}$ be the
  span of the $\omega_i(x,y)$ in $\Omega(K/C)$, with annihilator
  $\Ann(\Lambda) \subs \Der(K/C)$. The image of $\Lambda$ has
  codimension 1 in $\Omega(K/D_0)$, so $\Der(K/D_0) \cap
  \Ann(\Lambda)$ has dimension 1.  Let $D \in \Der(K/D_0) \cap
  \Ann(\Lambda)$ be nonzero. The image of $\Lambda$ spans
  $\Omega(K/F)$, so $\Der(K/F) \cap \Ann(\Lambda) = \{0\}$.  Hence
  $D\restrict{F} = \lambda D_0$ for some non-zero $\lambda$.
  Replacing $D$ by $\lambda^{-1}D$, we may assume that $\lambda=1$,
  that is, $D$ extends $D_0$. Indeed, we have shown that this $D$ is the unique derivation on $K$ extending $D_0$ such that $(x,y) \in \Gamma_S$ with respect to $D$.

  Let $K'$ be the differential closure of $\tuple{K;D}$, and let $C_K$ be the field of constants in $K'$. Since $K$ is algebraically closed, $C_K \subs K$. We must show that $C_K = C$. Let $F'$ be the algebraic closure in $K'$ of $C_K \cup F$. Now $F \subs K'$ is an inclusion of differentially closed fields, and the theory $\mathrm{DCF_0}$ has quantifier elimination, so the inclusion is an elementary inclusion. Thus it preserves all formulas in the differential field language, and in particular all existential formulas in the language $\L_\Ss$. It follows that it is a strong embedding when considered as an embedding of the reducts to the language $\L_\Ss$. So $F \strong K'$, and hence $F \strong K$. Furthermore, $F' \strong K$ since $F'$ is obtained from $F$ just by adding new constants. Thus $\delta(x,y/F') \ge 0$. Let $H$ be the smallest algebraic subgroup such that $(x,y)$ lies in a $C_K$-coset of $TH$, say $\gamma\cdot TH$. Then
\[\dim H \le \td(x,y/F') \le \dim (V \cap \gamma \cdot TH)\]
because $F' \strong K$, and because $(x,y) \in V \cap \gamma \cdot TH$ which is defined over $F'$. But $V$ is perfectly rotund, so $\dim H > \dim (V \cap \gamma \cdot TH)$ unless $H = S$. Thus $\td(x,y/F') = n = \td(x,y/F)$, so $F = F'$ and $C_K = C$.

Now $(x,y)$ is generic in $V$ over $C$, which means that it does not lie in any proper subvariety of $V$ defined over $C$. Thus we have 
  \[K \models (\exists(x,y) \in TS)(\forall e \in Q(C))[(x,y) \in \Gamma_S \cap V \wedge (x,y) \notin W_e]\]
 and this sentence remains true in $K'$ because there are no new constants. Since $F$ is an elementary substructure of $K'$, it also satisfies the same sentence.
  Thus $F$ satisfies the EC property.
\end{proof}

We can now give criteria for a system of exponential differential equations to have a
solution in some differential field. The Schanuel property can be
viewed as a necessary condition for a system of differential equations
to have a solution, and the EC property gives a matching sufficient
condition.

  Let $F$ be a differentially closed field, let $S$ be a semiabelian variety defined over the constant subfield $C$, and let $V$ be a subvariety of $TS$. Firstly, we replace $V \subs TS$ by a homomorphic
  image $V' \subs TS'$ which is free, with $\loc_C V'$ absolutely free. If $V'$ is defined over $C$ then a necessary and sufficient condition for there to be a nonconstant point in
  $\Gamma_{S'} \cap V'$ in $F$ is for $V'$ to be strongly rotund.

  If $V'$ is not defined over $C$ then a sufficient condition for a
  point to exist is for $V'$ to be rotund. If in addition $\loc_C
  V'$ is strongly rotund then a nonconstant point exists. Any such point gives rise to a point in $\Gamma_S \cap V$ by taking an inverse image under the quotient map.

The reduct of a differentially closed field does not have quantifier
elimination in the language $\L_\Ss$, so there is no general necessary
and sufficient condition when $V$ is defined with non-constant
parameters. The theory $\mathrm{DCF}_0$ does have quantifier
elimination, so there must be a condition which depends on what other
differential equations the parameters satisfy.


\section{The first order theory}\label{fot section}

\subsection{The uniform Schanuel property}

The compactness theorem of first order model theory can be combined
with the Schanuel property to give a \emph{uniform} Schanuel
property. 

The algebraic subgroups of $\ga^n$ are uniformly definable by formulas
of the form $Mx = 0$, where $M$ ranges over the definable set of
matrices $\Mat_{n\cross n}$. In other words, the algebraic subgroups
form a parametric family in the sense of definition~\ref{parametric family}. However, for all other commutative algebraic groups the set of all algebraic subgroups is not uniformly definable, and for semiabelian varieties there are no infinite parametric families of algebraic subgroups at all. This lack of uniform definability in fact works in our favour.

We use the fibre condition of algebraic geometry, from \cite[page
77]{Shafarevich1}.
\begin{lemma}[Fibre Condition]\label{fibre condition}
  Let $(V_p)_{p \in P}$ be a family of algebraic varieties,
  parametrized over a constructible set $P$. Then for each $k \in \N$,
  the set of fibres $\class{p \in P}{\dim V_p \ge k}$ is a subvariety of $P$ and
  the set $\class{p \in P}{\dim V_p = k}$ is constructible. \qed
\end{lemma}

A similar result holds for the rank of the Jacobian matrix in a
differential field with finitely many commuting derivations. Indeed,
upon close examination the main part of the proof of the fibre condition
is more or less this result.
\begin{lemma}\label{rk Jac is definable}
  For each algebraic variety $V$ and for each natural number $k$, the
  set $\class{x \in V}{\rk \Jac(x) \le k}$ is positively definable in
  the language of differential fields, and the set $\class{x \in
    V}{\rk \Jac(x) = k}$ is definable.
\end{lemma}
\begin{proof}
  $V$ is made up of finitely many affine charts, so it is enough to
  consider $V$ to be affine. For each $x$ the Jacobian $\Jac(x)$ is an
  $r \times n$ matrix. Its rank is the largest $k$ such that there is
  a $k \times k$ minor matrix with non-zero determinant. Thus $\rk
  \Jac(x) \le k$ iff $\det M = 0$ for every minor matrix $M$ of size
  $k+1$. The determinant is a polynomial and there are only finitely
  many minors, so this finite conjunction of equations is a positive
  first order condition on a matrix in the field language. The entries
  in the Jacobian are terms in the differential field language, and so
  we have positive definability of $\rk \Jac(x) \le k$. The second
  part follows.
\end{proof}

\begin{theorem}[Uniform Schanuel property]\label{semiabelian USC}
  Let $F$ be a differential field of characteristic zero, with
  finitely many commuting derivations. Let $S$ be a semiabelian
  variety of dimension $n$, defined over the constant subfield $C$ of
  $F$. For each parametric family $(V_c)_{c \in P(C)}$ of subvarieties
  of $TS$, with $V_c$ defined over $\Q(c)$, there is a finite set
  $\HV$ of proper algebraic subgroups of $S$ such that for each $c \in
  P(C)$ and each $(x,y) \in \Gamma_S \cap V_c$, if $\dim V_c -
  \rk\Jac(x,y) = n - t$ with $t>0$, then there is $\gamma \in TS(C)$
  and $H \in \HV$ of codimension at least $t$ in $S$ such that $(x,y)$
  lies in the coset $\gamma \cdot TH$.
\end{theorem}
\begin{proof}
  The set
  \[\Phi_V = \class{((x,y),c) \in \Gamma_S \cross
    P(C)}{(x,y) \in V_c, \dim V_c - \rk\Jac(x,y) = n-t}\] is definable
  using lemmas~\ref{fibre condition} and \ref{rk Jac is definable}.
  The set of formulas
  \[((x,y),c) \in \Phi_V \wedge (\exists \gamma \in TS(C))[(x,y) \in
  \gamma \cdot TH]\]
  where $H$ ranges over all proper algebraic subgroups of $S$ of
  codimension at least $t$ is countable (as there are only countably
  many proper algebraic subgroups of $S$); in particular it is of
  bounded size. It is unsatisfiable by the Schanuel property, so by
  the compactness theorem some finite subset of it is unsatisfiable.
  This gives the finite set $\HV$.
\end{proof}
For definiteness, we choose $\HV$ to be a particular minimal finite
set of subgroups for each variety $V$. The compactness method gives no
information about the nature of $\HV$, beyond it being finite.

\begin{cor}\label{SP is first order}
  The SP axiom can be written as a first order axiom scheme in the
  language $\L_\Ss$.
\end{cor}
\begin{proof}
  For each variety $P$ and each parametric family $(V_p)_{p \in P}$ of
  algebraic subvarieties of $TS$, defined over $\Q$, take the axiom
 
 \[(\forall p \in P(C))(\forall g \in \Gamma_S \cap V_p)\left[\dim V_p
 \le \dim S \to \bigvee_{H \in \HV} q_H(g) \in
 T(S/H)(C)\right]\]
 where $\HV$ is the finite set of algebraic subgroups of $S$ given by
 theorem~\ref{semiabelian USC} and $q_H$ is the quotient map
 $\ra{TS}{q_H}{T(S/H)}$.
\end{proof}

\subsection{The Weak CIT}

We next give a purely algebraic result about the intersection of
subvarieties and algebraic subgroups of a semiabelian variety. 
 The proof here is in essence the same as the proof of Zilber,
but simplified by using the full Schanuel property for partial
differential fields rather than just ordinary differential fields, and
by separating off the statement and proof of the uniform Schanuel
property. 

\begin{defn}
  Let $U$ be a smooth irreducible algebraic variety, and let $V,W$ be
  subvarieties of $U$, with $V \cap W \neq \emptyset$. The
  intersection $V \cap W$ is said to be \emph{typical} (in $U$) iff
  \[\dim(V \cap W) = \dim V + \dim W - \dim U\]
  and \emph{atypical} iff
  \[\dim(V \cap W) > \dim V + \dim W - \dim U.\]
  Even if $V$ and $W$ are irreducible, the intersection $V \cap W$ may
  be reducible, and its components may have different dimensions. We
  say that a component $X$ of $V\cap W$ is \emph{atypical} iff 
  \[\dim X > \dim V + \dim W - \dim U.\]
  We also say that the \emph{degree of atypicality} is the difference
  \[\dim X - (\dim V + \dim W - \dim U).\]
\end{defn}
Note that the intersection is typical iff $\codim(V\cap W) = \codim V
+ \codim W$, and since $U$ is smooth the dimension of the intersection
cannot be less than the typical size (assuming the intersection is
nonempty).

\begin{theorem}[``Weak CIT'' for semiabelian varieties]\label{weak
    CIT} Let $S$ be a semiabelian variety defined over an
  algebraically closed field $C$ of characteristic zero. Let $(U_p)_{p
    \in P}$ be a parametric family of algebraic subvarieties of $S$.
  There is a finite family $\Jf{U}$ of proper algebraic subgroups of
  $S$ such that, for any coset $\kappa = a \cdot H$ of any algebraic
  subgroup $H$ of $S$ and any $p \in P(C)$, if $X$ is an atypical
  component of $U_p \cap \kappa$ with degree of atypicality $t$, then
  there is $J \in \Jf{U}$ of codimension at least $t$ and $s \in S(C)$
  such that $X \subs s \cdot J$.  

  Furthermore, we may assume that $X$ is a typical component of the
  intersection $(U_p \cap s \cdot J) \cap (\kappa \cap s \cdot J)$ in
  $s \cdot J$.
\end{theorem}

The weak CIT is a simple corollary of the uniform Schanuel property,
but as well as the fact that there are no parametric families of
subgroups of a semiabelian variety, we use the fact that the subgroups
of a vector group \emph{do} form a parametric family.
\begin{proof}
  Let $n = \dim S$ and define $\Lambda_{Ma} = \class{x \in LS}{Mx = a}$
  where $M$ is an $n\cross n$ matrix and $a \in LS$. So $\Lambda$ is the
  parametric family of all cosets of algebraic subgroups of $LS$.

  Suppose that $X$ is an atypical component of $U_p \cap \kappa$ with
  \[r = \dim X = (\dim U_p + \dim \kappa - \dim S) + t.\]

  Let $y$ be generic in $X$ over $C$ and let $D_1,\ldots,D_r$ be a
  basis of $\Der(C(y)/C)$. Then $\rk\Jac(y) = r$. Take $x \in LS(F)$
  with $F$ some differential field extension such that $(x,y) \in
  \Gamma_S$. Then $\rk\Jac(x,y) = \rk\Jac(y)$. Now $y \in \kappa$, a
  constant coset of the algebraic subgroup $H$ of $S$, so, by axiom U4
  (see also step 4 of the proof of proposition~\ref{Main Schanuel
  prop}), $x$ lies in a constant coset of $LH$. Thus $x$ lies in
  $\Lambda_{Ma}$ for a suitable choice of $M \in \Mat_{n\cross n}(C)$
  and $a \in LS(C)$, with $\dim \Lambda_{Ma} = \dim \kappa$. Let
  $V_{Ma,p} = \Lambda_{Ma} \cross U_p$. Then $(x,y) \in \Gamma_S \cap
  V_{Ma,p}$ and
  \[\dim V_{Ma,p} - \rk\Jac(x,y) = \dim \kappa + \dim U_p - \dim X = \dim
  S - t\] and so by theorem~\ref{semiabelian USC}, there is $s \in
  S(C)$ and an algebraic subgroup $J$ of $S$ of codimension at least
  $t$ from the finite set $\HV$ such that $y \in s \cdot J$. Since $y$
  is generic in $X$ over $C$ and $s \cdot J$ is defined over $C$, we
  have $X \subs s \cdot J$. Thus, in the notation of
  theorem~\ref{semiabelian USC}, we may take the finite set $\Jf{U}$
  to be $\mathcal{H}_{\Lambda \cross U}$.

  Now $X \subs s \cdot J$, so $s^{-1} \cdot X \subs J$. Thus $X$ is an
  atypical component of the intersection $(U_p \cap s \cdot J) \cap
  (\kappa \cap s \cdot J)$ in $s \cdot J$ iff $s^{-1} \cdot X$ is an
  atypical component of the intersection $(s^{-1} \cdot U_p \cap J)
  \cap (s^{-1} \cdot \kappa \cap J)$ in $J$. If so, we may inductively
  find a smaller subgroup $J' \subs J$ from a finite set and a point
  $s' \in J(C)$ such that $X \subs s' \cdot J'$. Thus, inductively, we
  may assume that $X$ is a typical component of $(U_p \cap s \cdot J)
  \cap (\kappa \cap s \cdot J)$ in $s \cdot J$.
\end{proof}

The special case of the theorem where $S$ is an algebraic torus can be restated in more
elementary, less geometric terms.
\begin{cor}\label{gm weak CIT}
  For each $n,d,r \in \N$, there is $N \in \N$ with the following
  property. Suppose that $x = (x_1,\ldots,x_n) \in (\C^*)^n$ lies in
  an algebraic variety $U$ defined by $r$ polynomials of degree at
  most $d$, with coefficients in a subfield $K$ of $\C$. Suppose also
  that $x$ satisfies $l$ multiplicative dependencies of the form
  $\prod_{i=1}^n x_i^{m_{ij}} = a_j$ with the $m_{ij} \in \Z$ and $a_j
  \in K$, and that $\td(K(x_1,\ldots,x_n)/K) = \dim U - l + t$, with
  $t > 0$.

  Then $x$ satisfies $t$ multiplicative dependencies with the powers
  $m_{ij}$ having modulus at most $N$ and the $a_j$ lying in
  $\bar{K}$.
\end{cor}
\begin{proof}
  The subvarieties $U$ of $\gm^n$ defined by $r$ polynomials of degree
  at most $d$ can be put into a single parametric family. Take $C =
  \bar{K}$ in \ref{weak CIT}.
\end{proof}

This statement for tori has independently been reproved by Bombieri,
Masser, and Zannier in \cite{BMZ07}. They also use Ax's theorem (the
Schanuel property for the exponential equation) but use a heights
argument rather than the compactness theorem to get the natural number
$N$. This gives them an explicit bound which cannot be obtained
directly from the compactness theorem. Masser has noted in a private
communication to me that their method should also extend to the
semiabelian case.

\subsection{Definability of rotundity}

We generalize and adapt the proof in section 3 of
\cite{Zilber05peACF0} to show that rotundity is a definable property
of a variety. As well as the notion of an atypical intersection, we
also need the notion of an atypical image of a variety under a map, in
the context of subvarieties of groups.

\begin{defn}
  Let $G$ be an algebraic group, $H$ an algebraic subgroup and $V$ an
  algebraic subvariety of $G$. Let $\ra{G}{q}{G/H}$ be the quotient
  map onto the coset space and write $V/H$ for the image of $V$ under
  $q$. This image $V/H$ is said to be \emph{typical} iff
  \[\dim V/H = \min\{\dim G/H, \dim V\}\]
  and \emph{atypical} iff
  \[\dim V/H < \min\{\dim G/H, \dim V\}.\]
\end{defn}

We use the fact that in the conclusion of theorem~\ref{weak CIT},
$X$ is a typical component of the intersection $(U_p \cap s \cdot H)
\cap (\kappa \cap s \cdot H)$ in $s \cdot H$. For convenience we
also choose the finite set $\Jf{W}$ of subgroups of $S$ given in
the conclusion of that theorem to contain the trivial subgroup. The
additive formula for fibres is used several times:
\begin{itemize}
\item[(AF)] For an irreducible variety $A$ and a surjective
  map $\ra{A}{f}{B}$,
\[\dim A = \dim B + \min_{b \in B} \dim f^{-1}(b).\]
\end{itemize}

\begin{theorem}\label{finitary nature of rotundity}
  Let $S$ be a semiabelian variety and $V \subs TS$ an
  irreducible subvariety. If $V$ is not rotund then there is $J \in
  \Jf{W}$ where $W = \pr_S V$ such that $\dim V/TJ < \dim S/J$.
  That is, failure of rotundity is witnessed by a member of the finite
  set $\Jf{W}$.
\end{theorem}

\begin{proof}
  Suppose that $\dim V/TH < \dim S/H$ for some algebraic subgroup $H$
  of $S$. If $H = 1$ is the trivial subgroup then we are done since $1
  \in \Jf{W}$, so we assume that $\dim V \ge \dim S$, and $H \neq 1$.

  \paragraph{Step 1} The image $W/H$ is atypical.

  $W/H$ is a projection of $V/TH$, so 
  \[\dim W/H \le \dim V/TH < \dim S/H.\]
  Thus if $W/H$ were typical we would have $\dim W/H = \dim W$, so the
  fibres of the map $\ra{W}{}{W/H}$ would be finite. The fibres of
  $\ra{V}{}{V/TH}$ could then have dimension at most $\dim H$, so
  \[\dim V/TH \ge \dim V - \dim H \ge \dim S - \dim H = \dim
  S/H\] 
  which contradicts the assumption. Thus $W/H$ is atypical.

  \paragraph{Step 2} There is $J \in \Jf{W}$ such that
  \begin{equation}\label{eqnE}\dim W/J = \dim W/H - \dim J/(J \cap H)
  \end{equation}
  and
  \begin{equation}\label{eqnF}\dim W/H = \dim W/(J \cap H).
  \end{equation}

  Let $x \in W$ be generic over a field of definition of $S,H$ and
  $W$, and let $\kappa$ be the coset $x\cdot H$. Then $W \cap \kappa$
  is a generic fibre of the quotient map so, by the addition formula
  for fibres (AF),
  \[\dim W \cap \kappa = \dim W - \dim W/H\]
  which is strictly positive as the image is atypical. Let $X$ be the
  component of $W \cap \kappa$ containing $x$, which must be of
  maximal dimension by genericity of $x$. Thus
  \begin{equation}\label{eqnA}\dim X = \dim (W \cap \kappa) = \dim W
    - \dim W/H \end{equation} 
  and by atypicality of the image
  \[\dim W/H < \dim S/H = \dim S - \dim H\]
  so 
  \[\dim X > \dim W + \dim H - \dim S.\]
  Now $\dim H = \dim \kappa$ so $X$ is an atypical component of the
  intersection $W \cap \kappa$ in $S$. By theorem~\ref{weak CIT}
  there is $J \in \Jf{W}$ such that $X$ is contained in the coset
  $\kappa' = x \cdot J$. Thus the quotient of $X$ by $J \cap H$ is
  isomorphic to the quotient by $H$, so since $X$ is a component of
  maximal dimension this implies (\ref{eqnF}).

  By the remark above, $X$ is a typical component of $(W \cap \kappa')
  \cap (\kappa \cap \kappa')$ in $\kappa'$, that is
  \begin{equation}\label{eqnB}\dim X = \dim(W \cap \kappa') +
    \dim(\kappa \cap \kappa') - \dim \kappa'.\end{equation}
  Let $Y$ be the connected component of $(W \cap \kappa')$ containing
  $X$. Then (\ref{eqnB}) becomes
  \begin{equation}\label{eqnC}\dim X = \dim Y + \dim(J \cap H) - \dim
    J. \end{equation}
  $Y$ is a generic fibre of $\ra{W}{}{W/J}$, so by (AF) again,
  \begin{equation}\label{eqnD}\dim Y = \dim W - \dim J.\end{equation}
  Substituting (\ref{eqnA}) and (\ref{eqnD}) into (\ref{eqnC}) gives
  (\ref{eqnE}) as required.

  Let $H' = J \cap H$.

  \paragraph{Step 3} $\dim V/TH' < \dim S/H'$.

  For $b \in W$ write $V_b \subs LS$ for the fibre of the projection
  $\ra{V}{}{W}$. The projection $\ra{LS/LH'}{}{LS/LH}$ has fibres of
  dimension $k = \dim S/H' - \dim S/H$, so for any $b$ the fibres of
  the map $\ra{V_b/LH'}{}{V_b / LH}$ have dimension at most $k$. Thus
  \begin{equation}\label{eqnH}\dim V_b / LH' \le \dim V_b / LH + k.
  \end{equation}
    By (AF),
  \begin{equation}\label{eqnG} \dim V / TH' = \dim W/H' +
    \min_{b\in W} \dim V_b / LH'\end{equation} and substituting in
  (\ref{eqnG}) using (\ref{eqnF}) and (\ref{eqnH}) gives
  \[\dim V/TH' \le \dim W/H + \min_{b \in W} \dim V_b / LH +
  k\] which by (AF) again implies
  \[\dim V/TH' \le \dim V/TH + k < \dim S/H'\]
  as required.

  \paragraph{Step 4} $\dim V / TJ < \dim S/J$.

  This is very similar to step 3. Since $H' \subs J$, the quotient
  factors as 
  \[V \rTo V/TH' \rTo V/TJ\] 
  so for any $b \in W$, 
  \begin{equation}\label{eqnI}\dim V_b/LJ \le \dim V_b/LH'.
  \end{equation}
  By (AF),
  \begin{equation} \dim V/TJ = \dim W/J + \min_{b\in W} \dim
    V_b/LJ \end{equation} 
  and using (\ref{eqnE}) and (\ref{eqnI}) this becomes
  \[\dim V/TJ \le \dim W/H' + \min_{b\in W} \dim V_b/LH' +
  (\dim S/J - \dim S/H').\] 
  Applying (AF) a final time with the conclusion of Step 3 gives
  \[\dim V/TJ < \dim S/J\]
  as required.
\end{proof}

\begin{cor}\label{EC is first order}
  The EC axiom can be written as a first order axiom scheme in the
  language $\L_\Ss$.
\end{cor}
\begin{proof}
For a parametric family $(V_p)_{p \in P}$ of subvarieties of $TS$, let $\Rot_V(p)$ be given by
\[V_p \mbox{ is irreducible } \;\&\; \dim V_p = \dim S \;\&\; \bigwedge\nolimits_{J \in \Jf{\pr_S V}} \dim
      V_p/TJ \ge \dim S/J .\]
By theorem~\ref{finitary nature of rotundity}, this says that $V_p$ is rotund, irreducible, and of dimension $n = \dim S$. 
 By \cite[lemma~3]{Hru92}, for any parametric family of varieties $(V_p)_{p
  \in P}$ there is a first order formula in $p$ expressing that $V_p$ is irreducible. Hence by lemma~\ref{fibre condition} and the finiteness of $\Jf{\pr_S V}$, there is a first-order formula in the language of fields expressing $\Rot_V(p)$.

By proposition~\ref{EC' implies EC}, EC and EC$'$ are equivalent, so in the statement of EC it is enough to consider perfectly rotund subvarieties. In fact perfect rotundity is not definable, but every perfectly rotund subvariety is irreducible and of dimension $n$, so it is enough to consider just these subvarieties.

  For each $S \in \Ss$ and each pair of parametric families $(V_p)_{p \in P}$, $(W_e)_{e \in Q(C)}$ of
  subvarieties of $TS$, the families defined over $C_0$, take the
  following axiom.
  \[\begin{split}&
(\forall p\in P)(\exists g \in TS)(\forall e \in Q(C))[\Rot_V(p)  \to \\
& \mbox{\hspace{5em}} [g\in\Gamma_S \cap V_p \wedge (g \notin W_e \vee \dim W_e \cap V_p = \dim V_p)]]
 \end{split}\]
For the $V_p$ which are irreducible, the last clause says that $g$ does not lie in any of the $W_e$ whose intersection with $V_p$ is a proper subvariety of $V_p$. Hence this scheme of first-order sentences captures the EC property.
\end{proof}

\subsection{The first order theory}

Recall that $T_\Ss$ is the $\L_\Ss$-theory axiomatized by the
algebraic axioms U1 --- U7 and the Schanuel property SP, which are
given on page~\pageref{universal theory}, together with the
existential closedness axiom EC and non-triviality NT, which are given
on page~\pageref{EC defn}.

\begin{theorem}\label{T_S is first order}
  For each set $\Ss$ of semiabelian varieties, the theory $T_\Ss$ is
  the complete first order theory of the reduct to the language
  $\L_\Ss$ of a differentially closed field.
\end{theorem}
\begin{proof}
  We have shown that axioms U1 --- U7 are first order in lemma~\ref{U
  axioms are first order}, that the Schanuel property is first order
  in corollary~\ref{SP is first order}, and that existential
  closedness is first order in corollary~\ref{EC is first order}. It
  is immediate that NT is a first order axiom. Hence $T_\Ss$ is a
  first order theory. Proposition~\ref{reduct satisfies U axioms}
  shows that the reduct satisfies U1 --- U7, corollary~\ref{reduct
    satisfies SP} says that it satisfies SP, and theorem~\ref{reduct
    satisfies EC} says that it satisfies EC. NT is immediate.

  Since $T_\Ss$ is a first order theory, the part of
  proposition~\ref{Ks is amalg cat} which states that $\Kso$ has only
  countably many objects and countably many extensions of each object shows that
  every completion of $T_\Ss$ is $\aleph_0$-stable, and so (since it has no finite models) has a
  countable saturated model.  Let $M$ be a countable saturated model
  of $T_\Ss$. By saturation, $\td(C/C_0)=\aleph_0$. For each $n \in \N$, there is a unique $n$-type of a
  $\cl$-independent $n$-tuple. All of these types are realised in $M$,
  and hence $M$ satisfies ID. We claim that $M$ satisfies SEC. 

  Let $S \in \Ss$, let $V \subs TS$ be a rotund subvariety, and $A \subs F$ a finitely generated field of definition of $V$. For each proper subvariety $W$ of $V$, defined over $A$, we may use the Rabinovich trick to replace $V \minus W$ by a some $V' \subs TS'$ for some larger $S' \in \Ss$ as follows. Let $\bar{x}$ be the coordinates (homogeneous coordinates if necessary) of the variety $TS$, and say that $W$ is given by the equations $f_i(\bar{x}) = 0$ for $i=1,\ldots,m$. Let $S_1,\ldots,S_m \in \Ss$ be nontrivial, and for each $i$ let $z_i$ be a coordinate of the Lie algebra $LS_i$. Let $S' = S \cross \prod_{i=1}^m S_i$ and let $V'$ be the subvariety of $TS'$ given by $\bar{x} \in V$ and the equations $f_i(\bar{x})z_i = 1$ for $i=1,\ldots,m$. Then $V'$ is a rotund subvariety of $TS'$. 

  If necessary, we may now intersect $V'$ with generic hyperplanes as in the proof of proposition~\ref{EC' implies EC} to ensure that $\dim V' = \dim S'$. We can regard a family $(W_e)_{e \in Q(C)}$ of proper subvarieties of $V$, defined over $C$, as a family of subvarieties of $V'$ via the obvious co-ordinate maps. Now by EC there is $h \in \Gamma_{S'} \cap V' \minus \bigcup_{e \in Q(C)} W_e$. By the definition of $V'$, the projection $g$ of $h$ to $TS$ lies in $\Gamma_S \cap (V \minus W \minus \bigcup_{e \in Q(C)} W_e$.
  Hence, by the $\aleph_0$-saturation of $M$, there is $g \in \Gamma_S \cap V$, generic in $V$ over $A \cup C$. That is, SEC holds in $M$.

Thus, by theorem~\ref{uniqueness of U}, $M$ is isomorphic to the Fraiss\'e limit $U$. So $T_\Ss$ has exactly one countable saturated model, so only one completion, and hence it is complete.
\end{proof}

We end with two simple observations about the theories $T_\Ss$.
\begin{prop}
  For each set $\Ss$, the theory $T_\Ss$ has Morley rank $\omega$. 
\end{prop}
\begin{proof}
  $T_\Ss$ is a reduct of DCF${}_0$, hence it has Morley rank at most
  $\omega$. It has the theory of pairs of algebraically closed fields
  as a reduct (which in fact is $T_\Ss$ for $\Ss = \{1\}$), which has Morley rank
  $\omega$, so $T_\Ss$ has Morley rank $\omega$.
\end{proof}

\begin{prop}
  If $\Ss$ and $\Ss'$ are distinct collections of semiabelian varieties, each closed under products, subgroups, quotients, and under isogeny, then $T_\Ss \neq T_{\Ss'}$. (They are theories in different languages, so we mean they have no common definitional expansion.) Furthermore, all the theories $T_\Ss$ are proper reducts of (expansions by constant symbols of) $\mathrm{DFC}_0$.
\end{prop}
\begin{proof}
  Let $F$ be an $\aleph_0$-saturated differentially closed field. Without loss of generality $\Ss \nsubseteq \Ss'$, so take $S \in \Ss \minus \Ss'$. Choose an absolutely free and strongly rotund subvariety $V$ of $TS$, of dimension $\dim S + 1$. Then $F$ contains a point $g \in
  \Gamma_S \cap V$ with $\grk(g) = \dim S$, by SEC for the reduct
  of $F$ to $\L_{\Ss}$.

By the Schanuel property, $g$ cannot be
  algebraically dependent on any point $h \in \Gamma_{S'}$ for any $S' \in
  \Ss'$. Hence $g$ has dimension $\dim S+1$ in the sense of the
  pregeometry of the reduct to $\L_{\Ss'}$, but only dimension 1 in the sense of the pregeometry of the reduct to $\L_\Ss$. Thus the theories $T_\Ss$ and $T_{\Ss'}$ are distinct reducts of $\mathrm{DCF}_0$. Since every set $\Ss$ can be extended to a larger set of semiabelian varieties, if necessary by extending the constant field $C$, $T_\Ss$ is a proper reduct of $\mathrm{DCF}_0$.
\end{proof}



\end{document}